\documentclass[reqno,11pt]{amsart}
\usepackage[T1]{fontenc}
\usepackage{amsmath,amsthm,amssymb}
\usepackage[margin=1in]{geometry}
\usepackage{enumitem}
\usepackage{booktabs,xcolor,graphicx}
\usepackage{stmaryrd}
\usepackage[hidelinks]{hyperref}
\usepackage{comment}
\usepackage[textsize=footnotesize,backgroundcolor=orange!20]{todonotes}
\usepackage{mathtools}

\newtheorem{theorem}{Theorem}[section]
\newtheorem*{theorem*}{Theorem}
\newtheorem{proposition}[theorem]{Proposition}
\newtheorem{lemma}[theorem]{Lemma}
\newtheorem{claim}[theorem]{Claim}

\newtheorem*{question*}{Question}

\theoremstyle{definition}
\newtheorem{definition}[theorem]{Definition}

\theoremstyle{remark}

\newcommand{\paren}[1]{\left( #1 \right)}

\newcommand{\wh}{\widehat}
\newcommand{\cal}{\mathcal}

\DeclareMathOperator{\tower}{tower}

\newcommand{\ol}{\overline}

\newcommand{\E}{\mathsf{E}}

\newcommand{\CC}{\mathbb{C}}

\newcommand{\EE}{\mathbb{E}}
\newcommand{\FF}{\mathbb{F}}

\newcommand{\RR}{\mathbb{R}}

\newcommand{\ZZ}{\mathbb{Z}}

\newcommand{\cF}{\mathcal{F}}

\setenumerate[0]{label=(\roman*), ref=\roman*}

\begin{document}

\title[Tower-type bounds for Roth's theorem with popular differences]{Tower-type bounds for \\ Roth's theorem with popular differences}

\author{Jacob Fox}
\thanks{Fox was supported by a Packard Fellowship and by NSF Career Award DMS-1352121.}
\address{Fox, Department of Mathematics, Stanford University, Stanford, CA, USA}
\email{jacobfox@stanford.edu}

\author{Huy Tuan Pham}
\address{Pham, Department of Mathematics, Stanford University, Stanford, CA, USA.}
\email{huypham@stanford.edu}

\author{Yufei Zhao}
\thanks{Zhao was supported by NSF Award DMS-1764176, a Sloan Research Fellowship, and the MIT Solomon Buchsbaum Fund.}
\address{Zhao, Department of Mathematics, Massachusetts Institute of Technology, Cambridge, MA, USA}
\email{yufeiz@mit.edu}

\maketitle

\begin{abstract}
Green developed an arithmetic regularity lemma to prove a strengthening of Roth's theorem on arithmetic progressions in dense sets. It states that for every $\epsilon > 0$ there is some $N_0(\epsilon)$ such that for every $N \ge N_0(\epsilon)$ and $A \subset [N]$ with $|A| = \alpha N$, there is some nonzero $d$ such that $A$ contains at least $(\alpha^3 - \epsilon) N$ three-term arithmetic progressions with common difference $d$.
	
We prove that the minimum $N_0(\epsilon)$ in Green's theorem is an exponential tower of 2s of height on the order of $\log(1/\epsilon)$. Both the lower and upper bounds are new. It shows that the tower-type bounds that arise from the use of a regularity lemma in this application are quantitatively necessary.
\end{abstract}

\section{Introduction} \label{sc:intro}

A celebrated theorem of Roth \cite{Ro53} states that for each $\alpha>0$ there is a least positive integer $N(\alpha)$ such that if $N \geq N(\alpha)$ and $A\subset [N]:=\{1,\ldots,N\}$ with $|A| \geq \alpha N$, then $A$ contains a three-term arithmetic progression. Over the past six decades, there has been great efforts made by many researchers toward understanding the growth of this function, and despite the introduction of important tools in these efforts, the growth of $N(\alpha)$ is still not well understood. The upper bound was  improved by Heath-Brown \cite{HB}, Szemer\'edi \cite{Sz90}, Bourgain \cite{Bo99,Bo08}, Sanders \cite{Sa12,Sa11}, and most recently Bloom \cite{Bl} (see also \cite{BlSi}). The lower bound of Behrend \cite{B} was recently improved a bit by Elkin \cite{El} (see also Green and Wolf \cite{GrW} for a shorter proof). The best known bounds are of the form $\alpha^{-\Omega\left(\log\left(\alpha^{-1}\right)\right)} \leq N(\alpha) \leq 2^{O\left(\alpha^{-1}\left(\log \alpha^{-1}\right)^4\right)}$.  

Szemer\'edi \cite{Sz75} extended Roth's theorem to show that any dense set of integers contains arbitrarily long arithmetic progressions. Szemer\'edi's proof developed an early version of Szemer\'edi's regularity lemma \cite{Sz76}, which gives a rough structural result for large graphs and is arguably the most powerful tool developed in graph theory. It roughly says that any graph can be equitably partitioned into a bounded number of parts so that between almost all pairs of parts, the graph behaves randomly-like. Szemer\'edi's proof of the regularity lemma gives an upper bound on the number of parts which is tower-type in an approximation parameter, which gives a seemingly poor bound for the various applications of the regularity lemma. For over two decades there was some hope that a substantially better bound might hold leading to better bounds in the many applications. This hope was shattered by Gowers \cite{Go97}, who proved that the bound on the number of parts in the regularity lemma must grow as a tower-type function. Further results improving on some aspects of the lower bound were obtained in \cite{CoFo12,FL17,MS16}. 

It has been a major program over the last few decades to find new proofs of the various applications of Szemer\'edi's regularity lemma and its variants that avoid using the regularity lemma and obtain much better quantitative bounds. This program, popularized by Szemer\'edi and others, has been  quite successful, leading to the development of powerful new methods, such as in Gowers' new proof of 
Szemer\'edi's theorem  \cite{Go01} which introduced higher order Fourier analysis \cite{Go01}, and in the resolution of many open problems in extremal combinatorics using the powerful probabilistic technique known as dependent random choice (see the survey \cite{FoSu11}). However, until now it was unclear if one could avoid using regularity methods and obtain much better bounds in all known applications of the regularity lemma. 

A simple averaging argument of Varnavides \cite{Va59} shows that not only is there at least one arithmetic progression in a subset of $[N]$ of density $\alpha$ with $N$ sufficiently large, but in fact it must contain a positive constant fraction $c(\alpha)$ of the three-term arithmetic progressions. It is not difficult to show that $c(\alpha)$ is rather small, with $c(\alpha)=\alpha^{\omega(1)}$. In fact, one can show that $c(\alpha)$ is closely related to $N(\alpha)$. This bound on the density of three-term progressions is much smaller than the random bound of $\alpha^3$ one gets by considering a random set of density $\alpha$. 

Green \cite{G} developed an arithmetic analogue of Szemer\'edi's regularity lemma and used it to prove the following theorem, which answered a question of Bergelson, Host and Kra \cite{BHK}. It shows that while the total number of three-term arithmetic progressions can be much smaller than the random bound, there is a nonzero $d$ for which the number of three-term arithmetic progressions with common difference $d$ is at least roughly the random bound. 

\begin{theorem}[Green's popular progression difference theorem \cite{G}]\label{Greenstheorem}
For each $\epsilon>0$ there is an integer $N_0$ such that for any $N \geq N_0$ and subset $A \subset [N]$ with $|A|=\alpha N$, there is a nonzero $d$ such that $A$ contains at least $(\alpha^3-\epsilon)N$ three-term arithmetic progressions with common difference $d$.  
\end{theorem}

Similar to the graph setting, the proof of the arithmetic regularity lemma gives a tower-type upper bound on the size of the partition. Green \cite{G} proved a tower-type lower bound for the arithmetic regularity lemma in the setting of vector spaces over $\mathbb{F}_2$, and Hosseini, Lovett, Moshkovitz, and Shapira \cite{HLMS} later improved the tower height to $\Omega(\epsilon^{-1})$. Green's proof of Theorem \ref{Greenstheorem} uses the arithmetic regularity lemma and consequently shows that $N_0$  can be taken to be an exponential tower of twos of height $\epsilon^{-O(1)}$. It was unknown if the tower-type bounds that come from any regularity lemma  application like Theorem \ref{Greenstheorem} are necessary. Our main theorem determines the growth of the minimum $N_0$ for which Theorem \ref{Greenstheorem} holds, showing that it is an exponential tower of twos of height $\Theta(\log (1/\epsilon))$. Let $\tower(m)$ denote an exponential tower of twos of height $m$.

\begin{theorem}\label{mainpaper}
	Let $N_0(\epsilon)$ denote the smallest choice of $N_0$ for which Theorem \ref{Greenstheorem} holds. 
    There exist absolute constants $c, C >0$ such that for all $0<\epsilon < 1/2$,
	\[
	\tower(c \log (1/\epsilon)) \le N_0(\epsilon) \le \tower(C \log (1/\epsilon)).
	\]
\end{theorem}

This result, and the considerably easier analogous result in vector spaces over a fixed finite field \cite{FPI,FPII}, are the first examples of regularity lemma applications that require the tower-type growth. 

For investigations of popular differences for other patterns, including recent results on higher dimensional patterns, see \cite{BHK,Berger,FSSSZ,GT,Mandache,SSZ}.

\subsection{Detailed statement of results}

Theorem \ref{mainpaper} comes in two parts, an upper bound and a lower bound. The upper bound is as follows. 

\begin{theorem}[Upper bound for intervals] \label{th:upper-Z}
There exists a constant $C >0$ such that the following is true. Let $\epsilon > 0$ and $N \ge \tower(C \log (1/\epsilon))$. For every $A \subset [N]$, setting $\alpha = |A|/N$,
there exists some positive integer $d$ such that $x,x+d,x+2d \in A$ for at least $(\alpha^3 - \epsilon)N$ many integers $x$.
\end{theorem}

The main part of the proof is an analogous result in abelian groups of odd order. 

\begin{theorem}[Upper bound for abelian groups] \label{th:upper-G}
There exists a constant $C>0$ such that the following is true. Let $\epsilon > 0$ and let $G$ be a finite abelian group of odd order with $|G| \ge \tower(C \log (1/\epsilon))$. For every $A \subset G$, setting $\alpha = |A|/|G|$, there exists some $d \in G \setminus\{0\}$ such that $x,x+d,x+2d \in A$ for at least $(\alpha^3 - \epsilon)|G|$ many values of $x \in G$.
\end{theorem}

For the lower bound for intervals, we prove the following result, which is somewhat stronger than that lower bound result claimed in  Theorem \ref{mainpaper}. 

\begin{theorem}[Lower bound for intervals] \label{th:lower-Z}
There exist constants $c, \alpha_0 > 0$ such that for every $0 < \alpha \le \alpha_0$, $0 < \epsilon \le \alpha^{12}$ and $N \le \tower(c \log(1/\epsilon))$, there exists $A \subset [N]$ with $|A| \ge \alpha N$ such that for every positive integer $d \le N/2$, one has $x, x+d, x+2d \in A$ for at most $(\alpha^3 - \epsilon)(N-2d)$ many integers $x$.
\end{theorem}

\noindent {\bf Organization.} In the next section, we introduce some helpful notation and preliminaries including some basic facts from discrete Fourier analysis. In Section \ref{overview}, we give an overview of the proof strategies for our results.  In Section \ref{sec:upper}, we prove Theorems \ref{th:upper-Z} and \ref{th:upper-G} giving the upper bound results. In Section \ref{sec:building-block}, we give some auxiliary results for the probabilistic lower bound construction. In Section \ref{sec:Product of groups}, we give a lower bound construction for groups which are the product of prime cyclic groups with fast growing order. In Section \ref{sec:Intervals}, we then use this construction as an important ingredient to obtain the lower bound construction in intervals.  

We often omit floor and ceiling signs when they are not crucial for clarity of presentation. 

\section{Notations and Preliminaries} \label{sec:Notation}

\vspace{0.1cm}
\noindent \textbf{Averaging and expectation.}
We use $\E$ to denote the averaging operator: given a function $f$ on a finite set $S$, denote
\[
\E f = \E_{x \in S} [f(x)] := \frac{1}{|S|} \sum_{x \in S} f(x).
\]
We may write $\E_x$ instead of $\E_{x \in S}$ if the domain of $x$ is clear from context (usually over a group).

The $L^p$ norms are defined in the usual way:
\[
\|f\|_p := \paren{\E[|f|^p]}^{1/p}.
\]

As our lower bound construction is probabilistic, we will also need to consider expectations of random variables, for which we use the usual notation $\EE$ for expectation (note the difference in font compared to the averaging operator $\E$).

\vspace{0.1cm}
\noindent \textbf{Fourier transform and convolutions.}
Given a finite abelian group $G$, let $\wh G$ denote its dual group, whose elements are characters of $G$, i.e., homomorphisms $\chi \colon G \to S^1 := \{ z \in \CC : |z| = 1\}$. The Fourier transform of $f \colon G \to \CC$ is a function $\wh f \colon \wh G \to \CC$ defined by
\[
\wh f (\chi) := \E [f \ol\chi] = \E_{x \in G} [f(x) \ol{\chi(x)}].
\]
We write $\chi^{1/2}$ to denote the character given by $x \mapsto \chi(x/2)$ (we will always work with odd abelian groups so that $x/2$ makes sense).

It is often convenient to explicitly identify the dual group $\wh G$ with $G$ (they are isomorphic for finite abelian groups). For example, for $f\colon \ZZ_N \to \CC$ and $r \in \ZZ_N$, we identify $r$ with the character $\chi_r(x)=e(xr/N)$ where  we use the standard notation for the complex exponential 
\[
e(t) := \exp(2\pi i t), \quad t \in \RR.
\]
Thus,
\[
\wh f(r) := \E_x[f(x) e(xr/N))].
\]
Likewise, for $f \colon \FF_p^n \to \CC$ and $r \in \FF_p^n$, we identify $r$ with the character $\chi_r(x)=e(x\cdot r/p)$, where $x \cdot r = x_1r_1 + \cdots + x_n r_n \in \FF_p$ is the dot product in $\FF_p^n$. Thus,
\[
\wh f(r) := \E_x [f(x)  e(x\cdot r/p)].
\]

Given two functions $f$ and $g$ on $G$, their convolution $f*g$ is defined by
\[
(f * g)(x) := \E_{y \in G} [f(y)g(x-y)].
\]

We recall several useful properties of the Fourier transform: 
\begin{align*}
f(x) &= \sum_{\chi \in \wh G} [\wh f(\chi) \chi(x)], &&\text{\small[Fourier inversion formula]}\\
\wh {f*g} &= \wh f \cdot \wh g, &&\text{\small[Convolution identity]}\\
\mathsf{E}_{x\in G}[f(x)\overline{g(x)}] &= \sum_{\chi\in \wh G} \wh f(\chi)\overline{\wh g(\chi)}, &&\text{\small[Plancherel's identity]}\\
\mathsf{E}_{x\in G}[\left|f(x)\right|^2] &= \sum_{\chi\in \wh G} \left|\wh f(\chi)\right|^2. &&\text{\small[Parseval's identity]}\\
\end{align*}

The Fourier transform is also fundamentally related to the count of $3$-APs (or the count of solutions to linear equations in general), as evident in the following key identity already used in the proof of Roth's theorem \cite{Ro53}. It can be easily shown by substituting the  Fourier coefficients and expanding.  
\begin{equation}
\mathsf{E}_{x,d\in G}[f(x)f(x+d)f(x+2d)] = \sum_{\chi \in \wh G} \wh f(\chi)^2 \wh f(\chi^{-2}). \label{eq:3-APviafourier}
\end{equation}

\vspace{0.1cm}

\noindent \textbf{Densities.} For an abelian group $G$ with odd order and a function $f:G\to [0,1]$, we define the \emph{density of $3$-APs of $f$} as $$\mathsf{E}_{x,d\in G}[f(x)f(x+d)f(x+2d)].$$ We define the \emph{density of $3$-APs with common difference $d$ of $f$} as $$\mathsf{E}_{x\in G}[f(x)f(x+d)f(x+2d)].$$ For a subset $A$ of $G$, when we say ``density of $3$-APs'' of $A$, we mean that of its indicator function $1_A$, and likewise for ``density of $3$-APs with common difference $d$'' of $A$.

Over the interval $[N]$, we have two possible notions for the density of $3$-APs with common difference $d$ of a function $f:[N]\to [0,1]$. One can define the density of $3$-APs with common difference $d$ of $f$ as $$\frac{\sum_{x\in [N-2d]}[f(x)f(x+d)f(x+2d)]}{N},$$ as used in Theorem \ref{th:upper-Z}. This defines the density of 3-APs with common difference $d$ as the average weight of the $3$-APs $(x,x+d,x+2d)$ for $x\in [N]$, setting the value of $f$ outside $[N]$ to $0$. The other possible definition of the density of $3$-APs with common difference $d$ of $f$ is $$\mathsf{E}_{x\in [N-2d]}[f(x)f(x+d)f(x+2d)]=\frac{\sum_{x\in [N-2d]}[f(x)f(x+d)f(x+2d)]}{N-2d},$$ as used in Theorem \ref{th:lower-Z}. In this case, we take the average only over $3$-APs supported in $[N]$. It is easy to see that the density from the second definition is always at least the density from the first definition. In particular, the upper bound using the first definition implies the upper bound using the second definition. Similarly, the lower bound using the second definition implies the lower bound using the first definition. Thus, we give the stronger result in each case.

\vspace{0.1cm}

\noindent \textbf{Constants.} We use $c > 0$ and $C > 0$ to denote small and large absolute constants, though their values may differ at every instance. One could imagine attaching a unique subscript to each appearance of $c$ and $C$.

\section{Overview of strategy}\label{overview}

In this section we sketch the proof ideas of our main theorems, starting with the upper bound (Theorems~\ref{th:upper-G} and \ref{th:upper-Z}) and then followed by the lower bound (Theorem \ref{th:lower-Z}). 

In both cases, we prove the ``functional'' versions of the theorems. That is, instead of working with subsets $A \subset G$, we work with functions $f \colon G \to [0, 1]$, which can also be viewed as subsets with weighted elements. A subset $A \subset G$ can be represented by its indicator function $1_A$. Conversely, given a function $f \colon G \to [0,1]$, we can produce from it a random subset $A \subset G$ obtained by putting each element $x \in G$ independently into $A$ with probability $f(x)$. The resulting $A$ has similar statistical properties compared to $f$ due to concentration. Working with functions affords us greater flexibility, which is convenient for both parts of the proofs.

\subsection{Upper bound}

Theorems~\ref{th:upper-Z} and \ref{th:upper-G} follow from the functional forms below by setting $f = 1_A$.

\begin{theorem}[Upper bound for intervals, functional version] \label{th:upper-Z-f}
There exists a constant $C>0$ such that the following holds. Let $\epsilon > 0$ and $N \ge \tower(C \log (1/\epsilon))$. Let $f \colon [N] \to [0,1]$ with $\E f = \alpha$. Then there exists $d \in G\setminus\{0\}$ such that 
\[
\E_{x \in [N]} [f(x)f(x+d)f(x+2d)] \ge \alpha^3 - \epsilon.
\]
\end{theorem}

\begin{theorem}[Upper bound for abelian groups, functional version] \label{th:upper-G-f}
There exists a constant $C>0$ such that the following holds. Let $\epsilon > 0$ and let $G$ be a finite abelian group of odd order with $|G| \ge \tower(C \log (1/\epsilon))$. Let $f \colon G \to [0,1]$ with $\E f = \alpha$. Then there exists $d \in G\setminus\{0\}$ such that 
\[
\E_{x \in G} [f(x)f(x+d)f(x+2d)] \ge \alpha^3 - \epsilon.
\]
\end{theorem}

Green~\cite{G} proved the above theorems with a slightly worse bound of $\tower( C (1/\epsilon)^{C})$ instead of $\tower(C\log(1/\epsilon))$. Let us first give a quick sketch of Green's approach. It is easier to first explain it for the finite field vector space setting $G = \FF_p^n$ with $p$ fixed.

Green begins by establishing a regularity lemma. Given $f \colon G \to [0,1]$, one finds a subspace $H$ of $G = \FF_p^n$ of codimension at most $\tower(C (1/\epsilon)^C)$ such that inside almost all translates of $H$, $f$ behaves ``pseudorandomly'' in the sense of having small Fourier coefficients (other than the principal ``zeroth'' Fourier coefficient that records the density). This subspace $H$ is obtained iteratively, similar to the standard energy-increment proofs of regularity lemmas: starting with $H_0 = G$, at each step one checks if $H_i$ has the desired properties, and if not, then one finds a bounded-codimensional subspace $H_{i+1}$ of $H_i$ witnessing the non-uniformity. Each step increases the ``energy'', or mean-squared density, by at least $\epsilon^{O(1)}$. As the energy can never exceed $1$, the process terminates after at most $(1/\epsilon)^{O(1)}$ steps.

Once we have the above bounded-codimensional subspace $H$, let $g$ be the function obtained from $f$ by averaging $f$ inside each translate of $H$. In other words, consider the convolution $g = f * \beta_H$, where $\beta_H$ denotes the averaging measure on $H$ (normalized so that $\E \beta_H = 1$). The regularity property of $H$, namely $f - g$ having small Fourier coefficients when restricted to most $H$-cosets, is enough to deduce that $f$ and $g$ have similar densities of 3-APs \emph{with common differences lying in $H$}, i.e.,
\[
\E_{x \in G, d \in H} [f(x) f(x+d)f(x+2d)] \approx 
\E_{x \in G, d \in H} [g(x) g(x+d)g(x+2d)].
\]
On the other hand, $g$ is constant along $H$-cosets, so that $g(x) = g(x+d) = g(x+2d)$ for all $d \in H$. Thus the final expression is $\E [g^3] \ge (\E g)^3$ by convexity, and we have $\E g \approx \E f$. Putting everything together, we have
\begin{equation}
	\label{eq:upper-sketch-subsp-ineq}
\E_{x \in G, d \in H} [f(x) f(x+d)f(x+2d)] \ge (\E f)^3 - (\epsilon/2).
\end{equation}
If $G$ is large enough, so that $H$ is large enough, then the above inequality implies that there is some nonzero common difference $d \in H$ so that 
\[
\E_{x \in G} [f(x) f(x+d)f(x+2d)] \ge (\E f)^3 - \epsilon,
\]
thereby showing that $d$ is a popular common difference.

\medskip

Let us now sketch how to improve the above bound to $\tower(C \log(1/\epsilon))$ in the finite field setting, which had been worked out in \cite{FPI}. Instead of finding an $H$ that regularizes the function $f$, we simply seek to satisfy the inequality \eqref{eq:upper-sketch-subsp-ineq}. One then shows that if \eqref{eq:upper-sketch-subsp-ineq} is violated, then we can find a bounded-codimensional subspace of $H$, via an application of the weak regularity lemma at a ``local'' level, so that the corresponding \emph{mean cubed density} of $f$ (after averaging along the subspace) nearly doubles at each step (instead of merely increasing by $\epsilon^{-O(1)}$), so that the iteration process must end after $O(\log(1/\epsilon))$ steps (instead of $(1/\epsilon)^{O(1)}$ steps). Once we obtain an $H$ satisfying \eqref{eq:upper-sketch-subsp-ineq}, the rest of the argument is essentially identical.

\medskip

For general abelian groups $G$, unlike in the field field setting $\FF_p^n$, the group might not have enough subgroups to run the above arguments. Instead, one uses \emph{Bohr sets}, which play an analogous role to subgroups. Bohr sets are defined in Section~\ref{sec:Bohr}, and their manipulations are much more delicate compared to subspaces, particularly as they do not have as nice closure properties. Green proved his arithmetic regularity lemma for general abelian groups using Bohr sets as basic structural objects. The strategy remains largely similar in spirit to the finite field setting, though more challenging at a technical level (e.g., the group cannot be partitioned into Bohr sets, unlike with subgroups). For instance, we obtain $g$ from $f$ by setting $g(x)$ to be a certain ``smooth average'' of $f$ around a certain carefully chosen Bohr neighborhood of $x$. While the values $g(x)$, $g(x+d)$, and $g(x+2d)$ are no longer necessarily identical, they are hopefully approximately the same if $d$ lies inside some Bohr neighborhood of $0$.  The rest of Green's argument is similar to the finite field vector space case.

To obtain the corresponding result for intervals, one considers embedding $[N]$ in $\ZZ_N$ and only consider Bohr sets whose elements $d$ are all small in magnitude.

\medskip

In order to improve the bound from $\tower((1/\epsilon)^{O(1)})$ to $\tower(O(\log(1/\epsilon)))$ for general groups and for intervals, we carefully execute a combination of the above ideas. New ideas are required to adapt the mean cube density increment argument from \cite{FPI} to Bohr sets due to complications that do not arise in the finite field setting. The proof is carried out in full detail in Section~\ref{sec:upper}.

\subsection{Lower bound} \label{sec:overview-lower-bound}

In this section, we give a brief overview of the proof of Theorem \ref{th:lower-Z}. We will deduce Theorem \ref{th:lower-Z} from its functional analogue given below, where we replace the subset $A$ by a function $f:[N]\to [0,1]$ with density $\alpha$ so that for any nonzero $d$, the density of $3$-APs with common difference $d$ of $f$ is at most $\alpha^3(1-\epsilon)$. 

\begin{theorem}[Lower bound for intervals, functional version] \label{th:Interval-1}There are positive
absolute constants $c,\alpha_{0}$ such that the following holds.
If $0\le\alpha\le\alpha_{0}$, $0\le\epsilon\le \alpha^{7}$, and $N\le \tower(c\log(1/\epsilon))$, then there is a function $f:[N]\to[0,1]$ with $\mathsf{E}[f]=\alpha$ such
that for any integer $0<d<N/2$, 
\[
\mathsf{E}_{x\in[N-2d]}[f(x)f(x+d)f(x+2d)]\le\alpha^{3}(1-\epsilon).
\] 
\end{theorem}

We remark that we have replaced $\epsilon$ by $\epsilon \alpha^3$, which is more convenient to work with in the lower bound construction. Since we treat $\alpha$ as a constant throughout, this has no effect on the behavior of the asymptotic bound we get. The proof of Theorem \ref{th:lower-Z} assuming Theorem \ref{th:Interval-1} follows from a standard sampling argument, which we defer to Appendix \ref{appendix:lower}. There, we also show that it suffices to prove Theorem \ref{th:lower-Z} and Theorem \ref{th:Interval-1} for $N \ge \epsilon^{-15}$. 

In the following subsections, we sketch the construction of the function $f$ in Theorem \ref{th:Interval-1}. This construction utilizes a construction over cyclic groups which can be factored into a product of groups with appropriate growth in size. The construction in this case is inspired by the recursive construction presented in \cite{FPI} over finite field vector spaces. However, several important new ideas are needed. The construction of $f$ for such product groups is sketched in Subsection \ref{subsub:product}. Using this construction, we can construct $f$ over intervals, using ideas sketched in Subsection \ref{subsub:interval}, thus proving Theorem \ref{th:Interval-1}.

We remark that in this section we only give proof sketches without complete detail. The details of each construction and full proofs are presented in Sections \ref{sec:Product of groups} and \ref{sec:Intervals}.

\subsubsection{Product groups} \label{subsub:product}

We first give the construction for groups which can be written as a product of appropriately growing cyclic groups of prime order. 

\begin{theorem}[Lower bound for product of growing prime cyclic groups] \label{th:lower-product}
Let $0<\alpha \le 1/4$, $0
<\epsilon \le 20^{-9}$, and $G=\mathbb{Z}_n$ where $n$ is a positive integer such that there exist distinct primes
$m_{1},\ldots,m_{s}$ with $s\le\log_{150}\left(\frac{\epsilon^{-1/4}\alpha^{6}}{8}\right)$ satisfying
\begin{itemize}
\item $n=\prod_{j=1}^{s}m_j$,
\item $\epsilon^{-1/3}/2<m_{1}\le\epsilon^{-1/3}$, and 
\item for $i\ge 2$, $n_{i-1}^{6}<m_{i}<\exp(2^{-1}\cdot64^{-2}\cdot150{}^{i-1}\epsilon^{1/4}n_{i-1})/2$ where $n_i = \prod_{j=1}^{i}m_j$.
\end{itemize}
Then, there exists a function
$f:G\to[0,1]$ with $\mathsf{E}[f]=\alpha$ such that for any $d\in G\setminus\{0\}$,
\[
\mathsf{E}_{x}[f(x)f(x+d)f(x+2d)]\le\alpha^{3}(1-\epsilon).
\]
Furthermore, $\mathsf{E}_{x}[f(x)^3] \le 3\alpha^3/2$ and there exists $\alpha' \in [\alpha,\alpha(1+\epsilon^{1/4})]$ such that $f(x)=\alpha'$ for at least a $3/4$ fraction of $x \in G$.
\end{theorem} 

Here, the product structure of $G$ and the bound on the growth of $m_i$ allow us to conduct an iterative construction. The basic framework of the construction builds on the construction in \cite{FPI}, which took place in the setting of $\FF_p^n$. 

We build functions $f_1, f_2, f_3, \dots$ in this order. 
Here the domain of $f_i$ is $Q_{i} = \prod_{j=1}^i \mathbb{Z}_{m_j}$.
We will maintain that $\E f_i = \alpha$, and that for every $d\in Q_{i}\setminus \{0\}$, we have 
\[
\mathsf{E}_{x\in Q_{i}}[f_{i}(x)f_{i}(x+d)f_{i}(x+2d)] \le (1- \epsilon)(\E f_i)^3 .
\]
The function $f_s$ thus gives the desired construction.

Suppose we have already constructed $f_{i-1}$. Here is how we construct $f_{i}:Q_i\to [0,1]$:
\begin{enumerate}
    \item Design some family $\cal F$ of functions $\mathbb{Z}_{m_i}\to [0,1]$.
    \item Choose some $M_i \subset Q_{i-1}$ (with some properties).
    \item For each $x \in M_i$, fill in values of $f_{i}$ on the coset $x + \ZZ_{m_i}$ using a random function chosen from $\cal F$.
    \item For each $x \notin M_i$, set all values of $f_{i}$ on the coset $x + \ZZ_{m_i}$ to be $f_{i-1}(x)$.
\end{enumerate}

We refer to this process as \emph{random modification} (the phrase ``perturbation'' was used in \cite{FPI}). We will show that with positive probability, the random function $f_i$ satisfies the desired properties.

The main difference from the finite vector space case is the choice of the family $\cal{F}$. In \cite{FPI}, we first choose $g:\mathbb{F}_p\to [0,1]$ to be a multiple of the indicator of an interval of length $2p/3$. We then choose the family $\cF$ to be functions of the form $g(x \cdot v)$ for some nonzero $v \in \mathbb{F}_p^{m_i}$.

Here, instead, we choose a nice model function $g:\mathbb{Z}_{m_i}\to [0,1]$ satisfying certain properties to be discussed later. We choose $\cal{F}$ to consist of functions $g_{a,b}:\mathbb{Z}_{m_i}\to [0,1]$ defined by $g_{a,b}(x)=g(ax+b)$, indexed by $a,b\in \mathbb{Z}_{m_i}$ with $a\ne 0$. 

We denote elements of $Q_i$ by $x = (x_1,x_2,\dots,x_i)$ where $x_j \in \mathbb{Z}_{m_j}$.
We write $x_{[i-1]} = (x_1, \dots, x_{i-1}) \in Q_{i-1}$.
We can prove that for any $x\in Q_i$ with $x_{[i-1]}\in M_i$, and $d\in Q_i \setminus \{0\}$ such that $d_{[i-1]} = 0$, 
$$
\mathbb{E}_{a,b}[g_{a,b}(x)g_{a,b}(x+d)g_{a,b}(x+2d)] = \mathsf{E}_{y,z \in \mathbb{Z}_{m_i}, z\ne 0} [g(y)g(y+z)g(y+2z)] \le (1-c)(\mathsf{E}g)^3,
$$
where $a,b$ vary uniformly over all elements of $\ZZ_{m_i}$ with $a \ne 0$, 
and $y,z$ vary uniformly over all elements of $\ZZ_{m_i}$ with $z \ne 0$.
The final inequality is due to a property of the model function $g$.
It then follows via concentration that  with positive probability, 
$$\mathsf{E}_{x\in Q_{i}}[f_{i}(x)f_{i}(x+d)f_{i}(x+2d)] \le (1- \epsilon)(\E f_i)^3,$$
for every $d \in Q_i \setminus \{0\}$ with $d_{[i-1]} = 0$.

We are left with the task of bounding the density of 3-APs with common difference $d$ where $d_{[i-1]} \ne 0$. Over $\FF_p^n$, this is easy, as we can show, using the structure of vector spaces, that under a mild condition, if $d_{[i-1]} \ne 0$, then 
$$
    \mathsf{E}_{x\in Q_i}[f_i(x)f_i(x+d)f_i(x+2d)]=\mathsf{E}_{x\in Q_{i-1}}[f_{i-1}(x)f_{i-1}(x+d_{[i-1]})f_{i-1}(x+2d_{[i-1]})] \le \alpha^3(1-\epsilon).
$$ 
Such equality does not hold in our current setting.
Even though we do not need exact equality, obtaining uniform control over all $d\in Q_i$ with $d_{[i-1]}\ne 0$ using standard concentration inequalities does not work because $n_1$ is small compared to $\epsilon^{-1}$. 
However, we can indeed guarantee with high probability the equality 
$$
    \mathsf{E}_{x\in Q_i}[f_i(x)f_i(x+d)f_i(x+2d)]=\mathsf{E}_{x\in Q_{i-1}}[f_{i-1}(x)f_{i-1}(x+d_{[i-1]})f_{i-1}(x+2d_{[i-1]})],
$$ assuming that the model function $g$ over $\mathbb{Z}_{m_{i}}$ satisfies some additional nice properties, which we refer to as \textit{smoothness}. 

Roughly speaking, $g:\mathbb{Z}_{m_i}\to [0,1]$ is smooth if, for random $a_1,\ldots,a_h \in \mathbb{Z}_{m_i}\setminus \{0\}$, with high probability, one has
$$\mathsf{E}_y\left[\prod_{j=1}^h g(a_jy+b_j)\right] = \mathsf{E}[g]^h \quad \text{ for all }b_1,\ldots,b_h \in \ZZ_{m_i}.$$
To see how this smoothness property helps, assume that we are given $x' \in Q_{i-1}$ and $d'\in Q_{i-1} \setminus \{0\}$ such that 
\begin{align*}
f_i(x) &= g(a_1 x_i + b_1) \quad \text{for all $x$ such that $x_{[i-1]}=x'$},\\
f_i(x) &= g(a_2 x_i + b_2) \quad \text{for all $x$ such that $x_{[i-1]}=x'+d'$},\\
f_i(x) &= g(a_3 x_i + b_3) \quad \text{for all $x$ such that $x_{[i-1]}=x'+2d'$}.
\end{align*}
Then, for all $d$ with $d_{[i-1]}=d'$,  
$$
\mathsf{E}_{x\in Q_i:x_{[i-1]}=x'}[f_i(x)f_i(x+d)f_i(x+2d)] = \mathsf{E}_{y\in \mathbb{Z}_{m_i}}\left[g(a_1y+c_1)g(a_2y+c_2)g(a_3y+c_3)\right],
$$ where $c_1, c_2 ,c_3$ depend on $x'$ and $d$. 
Moreover, if $g$ is smooth, then with high probability over random $a_1,a_2,a_3$, one has
\begin{equation}
\mathsf{E}_{y\in \mathbb{Z}_{m_i}}\left[g(a_1y+c_1)g(a_2y+c_2)g(a_3y+c_3)\right] = \mathsf{E}[g]^3, \label{eq:smooth-triple-a}
\end{equation} for all $c_1,c_2,c_3 \in \ZZ_{m_i}$. In such case, we obtain that for all $d$ with $d_{[i-1]}=d'$,
$$
\mathsf{E}_{x\in Q_i:x_{[i-1]}=x'}[f_i(x)f_i(x+d)f_i(x+2d)] = \mathsf{E}[g]^3.
$$
Thus, as long as the parameters $a_1,a_2,a_3$ corresponding to each $x'\in Q_{i-1}$ and $d'\in Q_{i-1}\setminus\{0\}$ satisfy (\ref{eq:smooth-triple-a}), we obtain the desired equality 
$$
\mathsf{E}_{x\in Q_i}[f_i(x)f_i(x+d)f_i(x+2d)] = \mathsf{E}_{x\in Q_{i-1}}[f_{i-1}(x)f_{i-1}(x+d_{[i-1]})f_{i-1}(x+2d_{[i-1]})].
$$ We guarantee this property by applying the union bound over possible values of $d'$ and $x'$ in $Q_{i-1}$. Note that it is crucial here that the smoothness property allows us to avoid the union bound over $d_i$, which takes $m_i \approx \exp(O(|Q_{i-1}|))$ possible values.

The model function $g$ over $\mathbb{Z}_{m_{i}}$ is constructed in Section \ref{sec:building-block}. The essential idea behind the construction is that $g$ should be supported on only a few Fourier characters. The details of the construction over product groups are included in Section \ref{sec:Product of groups}.

\subsubsection{Intervals} \label{subsub:interval}
Using Theorem \ref{th:lower-product}, we can prove Theorem \ref{th:Interval-1}, giving the desired lower bound over intervals. The construction of the function $f$ in Theorem \ref{th:Interval-1} over intervals consists of three steps.

In the first step, we construct a function $f_1$ with density $\alpha$ which is $0$ in the interval $[N'+1,N]$ for $N'$ slightly smaller than $N$. This sets the density of 3-APs with common difference $d$ close to $N/2$ to $0$. 

In the second step, we let $f_2$ be the function obtained from the following procedure applied to $f_1$:
\begin{enumerate} 
\item Partition $[N']$ to $N'/q$ intervals $I_1,I_2,\dots,I_{N'/q}$ of length $q$ where $q$ can be written as a product of prime numbers as required in Theorem \ref{th:lower-product}. 
\item Using Theorem \ref{th:lower-product}, construct a function $g:\mathbb{Z}_q\to [0,1]$ which satisfies $\mathsf{E}[g]=\alpha$ and $\mathsf{E}_{x\in \mathbb{Z}_q} [g(x)g(x+d)g(x+2d)] \le \alpha^3(1-\epsilon)$ for any $d\in \mathbb{Z}_q \setminus \{0\}$.
\item For each $j=1,2,\dots,N'/q$, identify each interval $I_j$ with $\mathbb{Z}_q$ and place a copy of $g$ on each of them.
\end{enumerate}
For any $d$ with $0<d<N/2$ and $q \nmid d$, one can show that the density of $3$-APs with common difference $d$ of $f_2$ is at most $\alpha^3(1-\epsilon)$. However, for $d$ divisible $q$, the density of 3-APs with common difference $d$ of $f_2$ is larger than $\alpha^3$. 

Note that the function $f_2$ constructed in the second step is constant on each mod $q$ residue class in $[N']$. In the third step, we construct the function $f_3$ as follows:
\begin{enumerate}
    \item Construct a subset $X$ of $\mathbb{Z}_{N'/q}$ with much fewer 3-APs compared to the random bound using a variant of the Behrend construction.
    \item Let $P_t = \{x \in [N'] : x \equiv t \pmod q\}$. With some appropriate $T \subseteq \mathbb{Z}_q$, for each $t \in T$, take a random linear transformation of $X$ inside set $\mathbb{Z}_{N'/q}$, and set $f_3$ on $P_t$ to be the indicator function of this randomly transformed $X$.
    \item On $[N']\setminus \bigcup_{t\in T}P_t$, set $f_3$ to be equal to $f_2$. 
\end{enumerate}
The function $f_3$ has the property that in expectation, for a nonzero $d$ divisible by $q$, the density of 3-APs with common difference $d$ of $f_3$ is at most $\alpha^3(1-\epsilon)$.

We let $f=f_3$. Using concentration inequalities, we can show that with positive probability (over the randomness in the third step), for any $d \in [N/2]$, 
\[
\mathsf{E}_{x\in[N-2d]}[f(x)f(x+d)f(x+2d)]\le\alpha^{3}(1-\epsilon),
\] 
proving Theorem \ref{th:Interval-1}. The details of this construction are contained in Section \ref{sec:Intervals}.

\section{Upper bound} \label{sec:upper}

In this section, we prove Theorem \ref{th:upper-G} and Theorem \ref{th:upper-Z},
showing the existence of a popular difference for 3-APs when $|G|\ge{\rm tower}(C\log(1/\epsilon))$
or $N\ge{\rm {\rm tower}}(C\log(1/\epsilon))$. Here $G$ always denotes
a finite abelian group of odd order. For $x\in G$, we write $x/2$
to mean the inverse of the isomorphism $x\mapsto2x$. In Section
\ref{sec:Bohr}, we give some preliminaries on Bohr sets,
which is an important tool to make Fourier analysis work over general
abelian groups. In Section \ref{sec:upper-proofs}, we give the complete
proofs of Theorem \ref{th:upper-G} and Theorem \ref{th:upper-Z}. 

\subsection{Bohr sets \label{sec:Bohr}}

Denote the distance from $x\in\mathbb{R}$ to the nearest integer
by $\|x\|_{\mathbb{R}/\mathbb{Z}}:=\min_{n\in\mathbb{Z}}|x-n|$. Let
$\arg(z)$ denote the argument of $z\in\mathbb{C}$, so that $\arg(e^{it})\in[0,2\pi]$
and $e^{it}=e^{i\arg(e^{it})}$. 
\begin{definition}
Let $G$ be an abelian group of odd order. For a subset $S\subseteq\widehat{G}$
and a parameter $\rho\in[0,1]$, define the \textit{Bohr set} $B(S,\rho)=\{x\in G:\|\arg(\chi(x))/(2\pi)\|_{\mathbb{R}/\mathbb{Z}}\le\rho\:\forall\chi\in S\}$.
We call $S$ the \textit{frequency set} of the Bohr set $B(S,\rho)$
and $\rho$ the \textit{radius}. The \textit{codimension} of the Bohr
set is $|S|$. 
\end{definition}

We often drop $S$ and $\rho$ from the notation and denote the Bohr
set by $B$ if it is clear from context. Given a Bohr set $B$, we
write $S(B)$ to denote the frequency set of $B$. For a real number
$\nu\ge0$, we denote by $(B)_{\nu}$ the Bohr set with the same frequency
set and scaled radius $B(S,\nu\rho)$.

Define the normalized indicator function of Bohr sets by 
\[
\beta_{B}(x)=\frac{|G|B(x)}{|B|},
\]
where we chose the normalization so that $\mathsf{E}_{x}[\beta_{B}(x)]=1$.
Define 
\[
\phi_{B}(x)=\beta_{B}*\beta_{B}(x).
\]
Then we also have $\mathsf{E}_{x}[\phi_{B}(x)]=1$. The functions
$\beta_{B}$ can be thought of as the density with respect to the
uniform distribution on $G$ of the uniform distribution on $B$,
and $\phi_{B}$ is the density of a smoothened version of the uniform
distribution on $B$. In general, a function $\tau:G\to[0,\infty)$
with $\mathsf{E}\tau=1$ can be thought of as the density of a distribution
with respect to the uniform distribution on $G$. 

\textit{Conventions.} For simplicity of notation, we often omit the
subscript $B$ and use consistent subscripts throughout. For example,
$\phi=\phi_{B}$, $\beta=\beta_{B}$, $\phi_{1}=\phi_{B_{1}}$, $\beta_{1}=\beta_{B_{1}}$. 

As introduced by Bourgain~\cite{Bo99}, it is often useful to work with\textit{
regular Bohr sets}, those for which a small change to the radius does
not significantly change the size of the Bohr set. 
\begin{definition}
A Bohr set $B=B(S,\rho)$ of codimension $d$ is \textit{regular}
if for all $\delta\le1/(80d)$, 
\[
\left|(B)_{1+\delta}\setminus(B)_{1-\delta}\right|\le160\delta d|B|.
\]
\end{definition}

In the next proposition, we state some basic properties of Bohr sets,
whose proofs can be found in \cite[Section 4.4]{TV}. Denote by $2\cdot X=\{2x,x\in X\}$
the dilation of $X$ by a factor $2$. Recall that for a character
$\chi$, we denote by $\chi^{1/2}$ the character given by $x\mapsto \chi(x/2)$.

\begin{proposition}
\label{prop:properties-Bohr}The following properties hold.

\begin{enumerate}
\item $|B(S,\rho)|\ge|G|\rho^{|S|}$. 
\item $B(S,\rho)+B(S,\rho')\subseteq B(S,\rho+\rho')$. 
\item For every Bohr set $B=B(S,\rho)$, there exists $\nu\in[1/2,1]$ such
that $(B)_{\nu}$ is regular. 
\item For every Bohr set $B=B(S,\rho)$, $2\cdot B$ is also a Bohr set
with frequency set $\{\chi^{1/2}:\chi\in S\}$ and radius $\rho$. Furthermore $2\cdot B\subseteq(B)_{2}$. 
\item \label{enu:prop5-Bohr}For every Bohr set $B=B(S,\rho)$, $(2\cdot B)_{\nu}=2\cdot(B)_{\nu}$.
Hence, if $(B)_{\nu}$ is regular then $(2\cdot B)_{\nu}$ is also
regular. 
\end{enumerate}
\end{proposition}

\begin{definition}
Let $\phi$ be function on $G$ with $\mathsf{E}\phi=1$. Define $f_{\phi}(x)=(f*\phi)(x)=\mathsf{E}_{y}[f(x-y)\phi(y)]$. 
\end{definition}

The next estimate shows that regular Bohr sets are essentially invariant
under convolutions with a distribution whose support is contained
in a Bohr set with the same frequency set and smaller radius. This
is analogous to the additive closure property of subgroups. 
\begin{proposition}
\label{prop:continuity}Let $B$ be a regular Bohr set of codimension
$d$, $\beta=\beta_{B}$ and $\phi=\phi_{B}$. Let $\nu\le1/(80d)$.
Let $\tau$ be a function supported in $B_{\nu}$ with $\mathsf{E}\tau=1$.
Then 
\begin{equation}
\mathsf{E}_{x}[|(\beta*\tau)(x)-\beta(x)|]\le160\nu d,\label{eq:smooth-beta}
\end{equation}
and
\begin{equation}
\mathsf{E}_{x}[|(\phi*\tau)(x)-\phi(x)|]\le160\nu d.\label{eq:smooth-phi}
\end{equation}
Furthermore, for any $f:G\to[0,1]$, letting $\kappa$ be either $\beta$
or $\phi$, 
\begin{equation}
\mathsf{E}_{x}[|(f_{\tau}*\kappa)(x)-f_{\kappa}(x)|]\le160\nu d.\label{eq:smooth-f}
\end{equation}
\end{proposition}

\begin{proof}
We have 
\[
(\beta*\tau)(x)=\mathsf{E}_{y}[\beta(x-y)\tau(y)]\in[0,|G|/|B|].
\]
The support of $\beta*\tau$ is a subset of ${\rm supp}(\beta)+{\rm supp}(\tau)\subseteq B+(B)_{\nu}\subseteq(B)_{1+\nu}$.
Thus, if $x\notin(B)_{1+\nu}$, then $(\beta*\tau)(x)=0=\beta(x)$.
Furthermore, if $x\in(B)_{1-\nu}$, then for all $y\in{\rm supp}(\tau)$,
$x-y\in B$, so $(\beta*\tau)(x)=|G|/|B|=\beta(x)$. Hence, 
\[
\mathsf{E}_{x}[|(\beta*\tau)(x)-\beta(x)|]\le\frac{1}{|G|}\sum_{x\in(B)_{1+\nu}\setminus(B)_{1-\nu}}\frac{|G|}{|B|}=\frac{|(B)_{1+\nu}\setminus(B)_{1-\nu}|}{|B|}\le160\nu d,
\]
giving (\ref{eq:smooth-beta}).

For (\ref{eq:smooth-phi}), note that 
\begin{align*}
\mathsf{E}_{x}[|(\phi*\tau)(x)-\phi(x)|] & =\mathsf{E}_{x}[|(\beta*\beta*\tau)(x)-(\beta*\beta)(x)|]\\
 & =\mathsf{E}_{x}[|\mathsf{E}_{y}[\beta(y)((\beta*\tau)(x-y)-\beta(x-y))]|]\\
 & \le\mathsf{E}_{y}[\beta(y)\mathsf{E}_{x}[|(\beta*\tau)(x-y)-\beta(x-y)|]] &  & \textrm{(by the triangle inequality)}\\
 & \le160\nu d\mathsf{E}_{y}[\beta(y)] &  & \textrm{(by (\ref{eq:smooth-beta}))}\\
 & =160\nu d.
\end{align*}

For (\ref{eq:smooth-f}), we have 
\begin{align*}
\mathsf{E}_{x}[|f_{\tau}*\kappa-f_{\kappa}|] & =\mathsf{E}_{x}[|\mathsf{E}_{y}[f(y)(\kappa*\tau)(x-y)-f(y)\kappa(x-y)]|]\\
 & \le\mathsf{E}_{x,y}[|(\kappa*\tau)(x-y)-\kappa(x-y)|] &  & \textrm{(by the triangle inequality)}\\
 & \le160\nu d. &  & \textrm{(by (\ref{eq:smooth-beta}), (\ref{eq:smooth-phi}))}
\end{align*}
\end{proof}
The next lemma says that if $B_{2}\subseteq(B_{1})_{\nu/2}$ and $\tau=\beta_{B_2}$ or $\tau=\phi_{B_2}$, then
the $k$-th moment of $f_{\tau}$ is at least the $k$-th moment
of $f_{\phi_{B_1}}$ up to a small error term. 
\begin{lemma}
\label{prop:Jensen-Bohr}Let $f:G\to[0,1]$. Let $B_{1},B_{2}$ be
regular Bohr sets such that $B_{1}$ has codimension $d_{1}$. Let
$\phi_{1}=\phi_{B_{1}}$, $\beta_2=\beta_{B_2}$ and $\phi_{2}=\phi_{B_{2}}$. Let $k\ge 1$ be an integer and $\nu \le 1/(80d_1)$. Then following statements hold. 

If $B_2\subseteq (B_1)_{\nu/2}$, then
\begin{equation}
\mathsf{E}_{x}[f_{\phi_{2}}(x)^{k}]\ge\mathsf{E}_{x}[f_{\phi_{1}}(x)^{k}]-160\nu d_{1}k. \label{eq:jensen-phi}
\end{equation} If $B_2 \subseteq (B_1)_{\nu}$, then
\begin{equation}
\mathsf{E}_{x}[f_{\beta_2}(x)^{k}]\ge\mathsf{E}_{x}[f_{\phi_{1}}(x)^{k}]-160\nu d_{1}k. \label{eq:jensen-beta}
\end{equation}
\end{lemma}

\begin{proof}
By (\ref{eq:smooth-f}) of Proposition \ref{prop:continuity}, applied
with $B=B_{1}$, $\tau=\phi_{2}$ (noting that ${\rm supp}(\phi_{2})\subseteq B_{2}+B_{2}\subseteq(B_{1})_{\nu}$)
and $\kappa=\phi_{1}$, 
\[
\mathsf{E}_{x}\left[\left|f_{\phi_{1}}(x)-\mathsf{E}_{d}[f_{\phi_{2}}(x+d)\phi_{1}(d)]\right|\right]\le160\nu d_{1}.
\]
Thus,
\[
\mathsf{E}_{x}\left[\left|f_{\phi_{1}}(x)^{k}-(\mathsf{E}_{d}[f_{\phi_{2}}(x+d)\phi_{1}(d)])^{k}\right|\right]\le160\nu d_{1}k.
\]
By Jensen's inequality applied to the convex function $t\mapsto t^{k}$,
we obtain 
\begin{align*}
\mathsf{E}_{x}[f_{\phi_{1}}(x)^{k}] & \le\mathsf{E}_{x}[(\mathsf{E}_{d}[f_{\phi_{2}}(x+d)\phi_{1}(d)])^{k}]+160\nu d_{1}k\le\mathsf{E}_{x,d}[f_{\phi_{2}}(x+d)^{k}\phi_{1}(d)]+160\nu d_{1}k\\
 & =\mathsf{E}_{x}[f_{\phi_{2}}(x)^{k}]+160\nu d_{1}k.
\end{align*}

The proof of (\ref{eq:jensen-beta}) follows similarly. 
\end{proof}

\subsection{\label{sec:upper-proofs}Proofs of Theorem \ref{th:upper-G} and Theorem
\ref{th:upper-Z}}

In the following, we prove two results that are used in the proof
of Theorem \ref{th:upper-G}, the counting lemma (Lemma \ref{lem:counting-lemma})
and the mean-cube density increment (Lemma \ref{lem:cube-density-increment}). 

For a function $\phi$ on $G$ with $\mathsf{E}\phi=1$ and a function
$f:G\to[0,1]$, we denote 
\[
\Lambda_{\phi}(f)=\mathsf{E}_{x,d}[f(x)f(x+d)f(x+2d)\phi(d)].
\]

\begin{lemma}[Counting lemma]
\label{lem:counting-lemma}
Let $B_{1}=B(S_{1},\rho_{1})$ and $B_{2}=B(S_{2},\rho_{2})$
be two Bohr sets. 
Let $\phi_1 = \phi_{B_1}$ and $\phi_2 = \phi_{B_2}$.
Then 
\[
\Lambda_{\phi_{1}}(f_{\phi_{2}})\ge\Lambda_{\phi_{1}}(f)-3\sup_{\chi \in \widehat{G}}\left|\widehat{f}(\chi)-\widehat{f_{\phi_{2}}}(\chi)\right|\mathbb{E}[f(x)^{2}]\left(\frac{|G|}{|B_{1}|}\right)^{1/2}.
\]
\end{lemma}

\begin{proof}
By expanding in the Fourier basis, 
\[
\Lambda_{\phi_{1}}(f)=\sum_{\chi_{1}\chi_{2}\chi_{3}=1}\widehat{f}(\chi_{1})\widehat{f}(\chi_{2})\widehat{f}(\chi_{3})\widehat{\phi_{1}}(\chi_{2}^{-1}\chi_{3}^{-2}).
\]
By similar expansion for $\Lambda_{\phi_{1}}(f_{\phi_{2}})$, we can
write $\Lambda_{\phi_{1}}(f)-\Lambda_{\phi_{1}}(f_{\phi_{2}})$ as
\begin{align*}
 & \,\,\sum_{\chi_{1}\chi_{2}\chi_{3}=1}(\widehat{f}(\chi_{1})-\widehat{f_{\phi_{2}}}(\chi_{1}))\widehat{f}(\chi_{2})\widehat{f}(\chi_{3})\widehat{\phi_{1}}(\chi_{2}^{-1}\chi_{3}^{-2})\\
 & \qquad +\sum_{\chi_{1}\chi_{2}\chi_{3}=1}\widehat{f_{\phi_{2}}}(\chi_{1})(\widehat{f}(\chi_{2})-\widehat{f_{\phi_{2}}}(\chi_{2}))\widehat{f}(\chi_{3})\widehat{\phi_{1}}(\chi_{2}^{-1}\chi_{3}^{-2})\\
 & \qquad +\sum_{\chi_{1}\chi_{2}\chi_{3}=1}\widehat{f_{\phi_{2}}}(\chi_{1})\widehat{f_{\phi_{2}}}(\chi_{2})(\widehat{f}(\chi_{3})-\widehat{f_{\phi_{2}}}(\chi_{3}))\widehat{\phi_{1}}(\chi_{2}^{-1}\chi_{3}^{-2}).
\end{align*}

Note that 
\begin{align}
\left| \sum_{\chi_2 \chi_3 ^2 =\chi^{-1}} \widehat{f}(\chi_2)\widehat{f}(\chi_3)\right | &\le \left( \sum_{\chi_2} \left|\widehat{f}(\chi_2)\right|^2\right)\left(\sum_{\chi_3}\left|\widehat{f}(\chi_3)\right |^2\right) & & \textrm{(Cauchy-Schwarz)} \nonumber\\
& = \sum_{\chi} |\widehat{f}(\chi)|^2 \nonumber \\
& = \mathsf{E}[f(x)^2], & & \textrm{(Parseval)} \label{eq:bound23}
\end{align} and 
\begin{align}
\sum_{chi}\left |\widehat{\phi_1}\right| &\le \left(\sum_{\chi} \left |\widehat{\phi_1}\right|^2 \right )^{1/2} & & \textrm{(Cauchy-Schwarz)}  \nonumber \\
& = (\mathsf{E}_{x}[\phi_1(x)^2])^{1/2}. & & \textrm{(Parseval)} \label{eq:bound1}
\end{align}
We can now bound the first term as 
\begin{align*}
 & \left|\sum_{\chi_{1}\chi_{2}\chi_{3}=1}(\widehat{f}(\chi_{1})-\widehat{f_{\phi_{2}}}(\chi_{1}))\widehat{f}(\chi_{2})\widehat{f}(\chi_{3})\widehat{\phi_{1}}(\chi_{2}^{-1}\chi_{3}^{-2})\right|\\
 & \le\sup_{\chi}\left|\widehat{f}(\chi)-\widehat{f_{\phi_{2}}}(\chi)\right|\cdot\sum_{\chi}\left|\widehat{\phi_{1}}(\chi)\right|\left|\sum_{\chi_{2}\chi_{3}^{2}=\chi^{-1}}\widehat{f}(\chi_{2})\widehat{f}(\chi_{3})\right|\\
 & \le\sup_{\chi}\left|\widehat{f}(\chi)-\widehat{f_{\phi_{2}}}(\chi)\right|\cdot\sum_{\chi}\left|\widehat{\phi_{1}}(\chi)\right|\mathsf{E}[f(x)^{2}] &  & \textrm{(by (\ref{eq:bound23}))}\\
 & \le\sup_{\chi}\left|\widehat{f}(\chi)-\widehat{f_{\phi_{2}}}(\chi)\right|\mathsf{E}[f(x)^{2}]\left(\mathsf{E}_{x}[\phi_{1}(x)^{2}]\right)^{1/2} &  & \textrm{(by (\ref{eq:bound1}))}\\
 & \le\sup_{\chi}\left|\widehat{f}(\chi)-\widehat{f_{\phi_{2}}}(\chi)\right|\mathbb{E}[f(x)^{2}]\left(\frac{|G|}{|B_{1}|}\right)^{1/2},
\end{align*}
where in the last inequality, we used the fact that $\sup_{x}\phi_{1}(x)\le|G|/|B_{1}|$
and $\mathsf{E}\phi_{1}=1$. 

The two remaining terms are bounded similarly. 
\end{proof}

We next state and prove the mean-cube density increment lemma. We
make use of the following classical inequality in the proof of the
lemma: 

\begin{theorem}[Schur's inequality]\label{thm:schur}
For real numbers $a,b,c\ge0$, one has
\begin{equation}
a^{3}+b^{3}+c^{3}+3abc\ge a^{2}b+b^{2}a+a^{2}c+c^{2}a+b^{2}c+c^{2}b.
\end{equation}
\end{theorem}

\begin{lemma}[Mean-cube density increment]
\label{lem:cube-density-increment}Let $f:G\to[0,1]$. Let $B_{1}=B(S_{1},\rho_{1})$
and$
B_{2}=B(S_{2},\rho_{2})$
be two regular Bohr sets with codimension $d_{1},d_{2}$ respectively.
Let $\phi_{1}=\phi_{B_{1}}$ and $\phi_{2}=\phi_{B_{2}}$. Let $\nu\le1/(1000d_{1})$.
Assume that $B_2\subseteq (B_1)_{\nu^2/8} \cap (2\cdot B_1)_{\nu^2/8}$.
For every regular Bohr set $B=(B_1)_{\delta \nu/2}$ with
$\delta \in [1/2,1]$ (Proposition \ref{prop:properties-Bohr} guarantees the existence of such $\delta$ that makes $B$ regular), we have, setting $\phi=\phi_{B}$,
\[
\Lambda_{\phi}(f_{\phi_{2}})\ge2\mathsf{E}[f_{\phi_{1}}(x)^{3}]-\mathsf{E}[f_{\phi_{2}}(x)^{3}]-1920\nu d_{1}.
\]
\end{lemma}

\begin{proof}
Let $\beta=\beta_{B}$
and $\phi=\phi_{B}$. We denote by $\tilde{\beta}{}$
the normalized measure associated with the Bohr set $2\cdot B$,
and denote $\tilde{\phi}{}=\tilde{\beta}*\tilde{\beta}$. 

Applying Schur's inequality (Theorem~\ref{thm:schur}) with $a=f_{\phi_{2}}(x),b=f_{\phi_{2}}(x+d),c=f_{\phi_{2}}(x+2d)$
for each $x$ and $d$, and using linearity of expectation, we have
\begin{align*}
\Lambda_{\phi}(f_{\phi_{2}}) & =\mathsf{E}_{x,d}[f_{\phi_{2}}(x)f_{\phi_{2}}(x+d)f_{\phi_{2}}(x+2d)\phi(d)]\\
 & \ge\frac{4\mathsf{E}_{x,d}[f_{\phi_{2}}(x)^{2}f_{\phi_{2}}(x+d)\phi(d)]+2\mathsf{E}_{x,d}[f_{\phi_{2}}(x)^{2}f_{\phi_{2}}(x+2d)\phi(d)]}{3}-\mathsf{E}_{x}[f_{\phi_{2}}(x)^{3}].
\end{align*}

By Property (\ref{enu:prop5-Bohr}) in Proposition \ref{prop:properties-Bohr},
since $B$ is a regular Bohr set, $2\cdot B{}$ is also
a regular Bohr set. The codimensions of $2\cdot B{}$ and $B$
are $d_{1}$. We have that $B_{2}\subseteq (B_1)_{\nu^2/8}\cap (2\cdot B_1)_{\nu^2/8} \subseteq (B)_{\nu/2}\cap(2\cdot B)_{\nu/2}$.
By (\ref{eq:smooth-f}) in Proposition \ref{prop:continuity}, applied
with the Bohr set $B$, $\kappa=\phi,\tau=\phi_{2}$, 
\[
\mathsf{E}_{x}\left[\left|\mathsf{E}_{d}[f_{\phi_{2}}(x+d)\phi(d)]-f_{\phi}(x)\right|\right]\le160\nu d_{1}.
\]
Hence,
\begin{align*}
\mathsf{E}_{x,d}[f_{\phi_{2}}(x)^{2}f_{\phi_{2}}(x+d)\phi(d)] & \ge\mathsf{E}_{x}[f_{\phi_{2}}(x)^{2}f_{\phi}(x)]-160 \nu d_{1}.
\end{align*}
Similarly,
\[
\mathsf{E}_{x,d}[f_{\phi_{2}}(x)^{2}f_{\phi_{2}}(x+2d)\phi(d)]\ge\mathsf{E}_{x}[f_{\phi_{2}}(x)^{2}f_{\tilde{\phi}}(x)]-160 \nu d_{1}.
\]

We have 
\begin{align*}
\mathsf{E}_{x}[f_{\phi_{2}}(x)^{2}f_{\phi}(x)] 
&= \mathsf{E}_{x,y}[f_{\phi_{2}}(x)^{2}\beta(y) f_{\beta}(x-y)] \\
&= \mathsf{E}_{x,y}[f_{\phi_{2}}(x)^{2}\beta(y) f_{\beta}(x+y)] & & \textrm{(using $\beta(-y)=\beta(y)$)}\\
&= \mathsf{E}_{x}[f_{\beta}(x)\mathsf{E}_{y}[f_{\phi_{2}}(x-y)^{2}\beta(y)]]\\
&\ge \mathsf{E}_{x}[f_{\beta}(x)\mathsf{E}_{y}[f_{\phi_{2}}(x-y)\beta(y)]^2]. & & \textrm{(Cauchy-Schwarz)}
\end{align*}
By (\ref{eq:smooth-f}) in Proposition \ref{prop:continuity}, applied
with the Bohr set $B$, $\kappa=\beta,\tau=\phi_{2}$, 
\[
\mathsf{E}_{x}\left[\left|\mathsf{E}_{y}[f_{\phi_{2}}(x-y)\beta(y)]-f_{\beta}(x)\right|\right]\le160\nu d_{1}.
\]
Thus, 
\begin{align*}
\mathsf{E}_{x}[f_{\beta}(x)\mathsf{E}_{y}[f_{\phi_{2}}(x-y)\beta(y)]^2] 
&\ge \mathsf{E}_{x}[f_{\beta}(x)^3] - 320\nu d_1. 
\end{align*} 
Hence, 
\[
\mathsf{E}_{x}[f_{\phi_{2}}(x)^{2}f_{\phi_2}(x+d)\phi(d)]  \ge \mathsf{E}_{x}[f_{\beta}(x)^3] - 480 \nu d_1.
\]

Similarly, noting that $B_2\subseteq (2\cdot B)_{\nu/2}$, we have
\begin{align*}
\mathsf{E}_{x,d}[f_{\phi_{2}}(x)^{2}f_{\phi_{2}}(x+2d)\phi(d)] &= \mathsf{E}_{x,d}[f_{\phi_{2}}(x)^{2}f_{\phi_{2}}(x+d)\tilde{\phi}(d)] \ge \mathsf{E}_{x}[f_{\tilde{\beta}}(x)^3] - 480 \nu d_1.
\end{align*}

Since $B\subseteq (B_1)_{\delta \nu}$, by (\ref{eq:jensen-beta}) of Lemma \ref{prop:Jensen-Bohr} applied to the Bohr sets $B_1$ and $B$, 
\[
 \mathsf{E}_{x}[f_{\beta}(x)^3]  \ge \mathsf{E}_{x}[f_{\phi_1}(x)^3] - 480 \nu d_1.
\] Thus, $$\mathsf{E}_{x}[f_{\phi_{2}}(x)^{2}f_{\phi_2}(x+d)\phi(d)]  \ge  \mathsf{E}_{x}[f_{\phi_1}(x)^3] - 960 \nu d_1.$$ 

Similarly, since $2\cdot B\subseteq (B)_2 \subseteq (B_1)_{\delta \nu}$, by (\ref{eq:jensen-beta}) of Lemma \ref{prop:Jensen-Bohr} applied with the Bohr sets $B_1$ and $2\cdot B$, $$\mathsf{E}_{x}[f_{\tilde{\beta}}(x)^3] \ge \mathsf{E}_{x}[f_{\phi_1}(x)^3] - 480\nu d_1.$$ Thus, \[
\mathsf{E}_{x}[f_{\phi_{2}}(x)^{2}f_{\phi_2}(x+2d)\phi(d)]  \ge \mathsf{E}_{x}[f_{\phi_1}(x)^3] - 960 \nu d_1.
\]

Combining, we get
\[
\Lambda_{\phi}(f_{\phi_{2}})\ge2\mathsf{E}[f_{\phi_1}^{3}]-\mathsf{E}[f_{\phi_{2}}^{3}]-1920\nu d_1. \qedhere 
\]
\end{proof}

\begin{lemma}
\label{lem:incr-numbers}Let $\alpha,\epsilon>0$. Let $a_{1},a_{2},\dots$
be a sequence of positive real numbers such that $\alpha^{3}\le a_{i}\le1$
for all $i$. Then for some $i\le2\log_{2}(2/\epsilon)$, $2a_{i}-a_{i+1}\ge\alpha^{3}-\epsilon/2$. 
\end{lemma}

\begin{proof}
Assume for the sake of contradiction that for all $i\le2\log_{2}(2/\epsilon)$,
\[
a_{i+1}\ge2a_{i}-\alpha^{3}+\epsilon/2.
\]
Then $a_{2}\ge\alpha^{3}+\epsilon/2$ since $a_{1}\ge\alpha^{3}$.
For $2\le i\le2\log_{2}(2/\epsilon)$, $a_{i+1}-\alpha^{3}\ge2(a_{i}-\alpha^{3})$,
so $a_{i+1}\ge\alpha^{3}+2^{i}\epsilon/2.$ Since $a_{i+1}\le1$ for
all $i$, we arrive at a contradiction since $2^{2\log_{2}(2/\epsilon)}\epsilon/2>1$. 
\end{proof}

\begin{proof}[Proof of Theorem \ref{th:upper-G}]
We define inductively parameters $\rho_{i}$ such that $\rho_{1}=\epsilon^{10}$,
and for $i\ge2$, $\rho_{i}=\exp(-\rho_{i-1}^{-5})$. Let $\nu_{i}=10^{-5}\epsilon \rho_{i}^{2}$.

Let $S_{1}=\{\chi\in\widehat{G}:|\widehat{f}(\chi)|\ge\rho_{1}/2\}$, and for $i\ge 2$, $S_{i} = \{\chi\in\widehat{G}:|\widehat{f}(\chi)|\ge\rho_{1}/2\} \cup \{\chi^{1/2}:\chi \in S_{i-1}\}$. 
By Parseval's identity, $|\{\chi\in\widehat{G}:|\widehat{f}(\chi)|\ge\rho_{1}/2\}|\le4\rho_{i}^{-2}$, so $|S_{i}| \le \sum_{j=1}^{i} 4\rho_{j}^{-2} < 5 \rho_{i}^{-2}$. Let $B_{i}=B(S_{i},\rho_{i}/(4\pi))$.
We first note that for $\chi\in S_{i}$, 
\[
\left|1-\widehat{\beta_{i}}(\chi)\right|=\left|1-\mathsf{E}_{x}[\beta_{i}(x)\chi(x)]\right|\le\mathsf{E}_{x}[\beta_{i}(x)\left|\chi(x)-1\right|]\le\rho_{i}/2,
\]
since $\|\arg(\chi(x))/(2\pi)\|_{\mathbb{R}/\mathbb{Z}}\le\rho_{i}/(4\pi)$
for all $x$ such that $\beta_{i}(x)\ne0$. Hence, 
\[
\left|1-\widehat{\phi_{i}}(\chi)\right|=\left|1-\widehat{\beta_{i}}(\chi)^{2}\right|\le\rho_{i}.
\]
Thus, for $\chi\in S_{i}$, 
\[
\left|\widehat{(f-f_{\phi_{i}})}(\chi)\right|\le\rho_{i},
\]
and for $\chi\notin S_{i}$, $|\widehat{f}(\chi)|\le\rho_{i}/2$ so
\[
\left|\widehat{(f-f_{\phi_{i}})}(\chi)\right|\le\rho_{i}.
\]

Observe that $\mathsf{E}[f_{\phi_{i}}^{3}]\ge\alpha^{3}$ by convexity,
and $\mathsf{E}[f_{\phi_{i}}^{3}]\le1$ for all $i$. By Lemma
\ref{lem:incr-numbers}, there exists $i\le2\log_{2}(2/\epsilon)$
such that 
\[
2\mathsf{E}[f_{\phi_{i}}(x)^{3}]-\mathsf{E}[f_{\phi_{i+1}}(x)^{3}]\ge\alpha^{3}-\epsilon/2.
\]

Fix such an $i$. We have $B_{i+1}\subseteq(B_{i})_{\nu_i^2/8}\cap (2\cdot B_{i})_{\nu_i^2/8}$
since for any $\chi\in S_{i}$ and $x\in B_{i+1}$, 
\[
\|\arg(\chi(x))/(2\pi)\|_{\mathbb{R}/\mathbb{Z}}\le\rho_{i+1}\le\nu_{i}^{2}\rho_{i}/8,
\]
and furthermore $\chi^{1/2}\in S_{i+1}$ so
\[
\|\arg(\chi^{1/2}(x))/(2\pi)\|_{\mathbb{R}/\mathbb{Z}}\le \rho_{i+1}\le\nu_{i}^{2}\rho_{i}/8.
\]
By Lemma \ref{lem:cube-density-increment}, there exists a regular Bohr set $B=(B_i)_{\delta \nu_i/2}$ for $\delta\in [1/2,1]$ such that
\[
\Lambda_{\phi}(f_{\phi_{i+1}})\ge2\mathsf{E}[f_{\phi_{i}}(x)^{3}]-\mathsf{E}[f_{\phi_{i+1}}(x)^{3}]-1920\nu_i |S_i|\ge\alpha^{3}-\epsilon/2-1920\nu_i |S_i|.
\]
By Lemma \ref{lem:counting-lemma}, 
\[
\Lambda_{\phi}(f_{\phi_{i+1}})\le\Lambda_{\phi}(f)+\sup_{\chi}\left|\widehat{f}(\chi)-\widehat{f_{\phi_{i+1}}}(\chi)\right|\mathsf{E}[f(x)^{2}]\left(\frac{|G|}{|B|}\right)^{1/2}\le\Lambda_{\phi}(f)+\rho_{i+1}(4\nu_i^{-1}\rho_i^{-1})^{5\rho_{i}^{-2}}.
\]
By our choice of $\rho_{i},\nu_{i}$, we have 
$$
\rho_{i+1}(4\nu_i^{-1}\rho_i^{-1})^{5\rho_{i}^{-2}}\le\exp(5\rho_{i}^{-2}\cdot\log(10^6 \epsilon^{-1}\rho_{i}^{-3}))\exp(-\rho_{i}^{-5})<\epsilon/8,
$$
noting that $\rho_{i}\le\rho_{1}\le\epsilon^{10}$. Hence, 
\[
\Lambda_{\phi}(f)\ge\alpha^{3}-\epsilon/2-1920\nu_{i}d_{i}-\rho_{i+1}(4\nu_i^{-1}\rho_i^{-1})^{5\rho_{i}^{-2}}>\alpha^{3}-\epsilon/2-\epsilon/4-\epsilon/8=\alpha^{3}-7\epsilon/8.
\]

Observe that there exists an absolute constant $C'>0$ so that $\rho_{i}\ge1/{\rm tower}(C'i)$. Furthermore, the codimension of $B_{i}$ is bounded above
by $5\rho_{i}^{-2}$ and the radius of $B_{i}$ is $\rho_i/(4\pi)$. Hence, we obtain a Bohr set $B$ with size at least
$|G|/{\rm tower}(10C'\log(1/\epsilon))$ such that $\Lambda_{\phi}(f)\ge\alpha^{3}-7\epsilon/8$.
Hence, for a sufficiently large constant $C>0$, assuming that $|G|\ge{\rm tower}(C\log(1/\epsilon))$, we have 
\[
\mathsf{E}_{x,d}[f(x)f(x+d)f(x+2d)\phi(d)I(d\ne0)]\ge\Lambda_{\phi}(f)-\frac{1}{|B|}\ge\alpha^{3}-\epsilon.
\]
Thus, there exists $d\ne0$ such that 
\[
\mathsf{E}_{x}[f(x)f(x+d)f(x+2d)]\ge\alpha^{3}-\epsilon.
\]
 
\end{proof}

\begin{proof}[Proof of Theorem \ref{th:upper-Z}]
We can assume without loss of generality that $N$ is odd by possibly
increasing $N$ by $1$. Let $G=\mathbb{Z}_N$. We repeat
the proof of Theorem \ref{th:upper-G} with the inclusion of the character
$\chi_{0}(x)=e^{2\pi ix/N}$ in the sets $S_{i}$. We then obtain
a Bohr set $B$ whose frequency set contains $\chi_{0}$ such that
$B$ has size at least $|G|/{\rm tower}(C'\log(1/\epsilon))$ and
$\Lambda_{\phi}(f)\ge\alpha^{3}-7\epsilon/8$. Assuming that $N\ge{\rm tower}(C\log(1/\epsilon))$
for sufficiently large $C$, following the last step in the proof
of Theorem \ref{th:upper-G}, we obtain a positive integer $d<N/2$
such that $d\in{\rm supp}(\phi)$ when viewed an element in $\mathbb{Z}_N$
and 
\[
\mathsf{E}_{x}[f(x)f(x+d)f(x+2d)]\ge\alpha^{3}-15\epsilon/16.
\]
Since $d\in{\rm supp}(\phi)\subseteq B+B$ and $\chi_{0}$ is in the
frequency set defining $B$, $\|\arg(\chi_{0}(d))/(2\pi)\|_{\mathbb{R}/\mathbb{Z}}\le2\rho\le2\epsilon^{10}$.
Thus, as a positive integer less than $N/2$, we have $d<2\epsilon^{10}N$.
Thus, restricting to $x\in[N-2d]$ in the above expectation, we have
\[
\sum_{x\in[N-2d]}f(x)f(x+d)f(x+2d)\ge N(\alpha^{3}-15\epsilon/16)-2d\ge N(\alpha^{3}-\epsilon).
\]
\end{proof}

\section{Lower bound construction: preparations} \label{sec:building-block}

We assume throughout this section that $N$ is an odd prime number. As a building block in our construction, we will make use of a function $g$ which has relatively low $3$-AP density (considerably smaller than the random bound given the density of $g$), but behaves random-like in the following way. If $a_{1},a_{2},\ldots,a_{h}$
are chosen independently and uniformly at random from the nonzero
elements of $\mathbb{Z}_{N}$, then with high probability, for all $b_{1},b_{2},\ldots,b_{h}\in\mathbb{Z}_{N}$,
\begin{equation}
\mathsf{E}_{x}\left[\prod_{j=1}^{h}g(a_{j}x+b_{j})\right]=\mathsf{E}[g]^{h}.\label{eq:smooth concentration}
\end{equation} 

In the following, we identify $\widehat{\mathbb{Z}_{N}}$ with $\mathbb{Z}_{N}$, so that we write
\[
\widehat g(r) = \mathsf{E}_{x\in \mathbb{Z}_N}\left[g(x)e\left(\frac{rx}{N}\right)\right].
\]

\begin{lemma} \label{lemma:smoothness} Suppose $g:\mathbb{Z}_{N}\to[0,1]$
and $a_{1},a_{2},\ldots,a_{h}\in\mathbb{Z}_{N}\setminus\{0\}$ satisfy
the following properties: 
\begin{enumerate}
\item \label{enu:cond2-smoothness}  The support of $\wh g$ has size at most $\ell$.
\item \label{enu:cond4-smoothness}  For all $r_{1},r_{2},\ldots,r_{h}\in\mathbb{Z}_{N}$ such that $\sum_{j=1}^{h}r_{j}a_{j}=0$
and $(r_{1},r_{2},\ldots,r_{h})\ne(0,0,\ldots,0)$, there is some $j \in [h]$ such that $r_{j}$ is not contained in the support of $\wh g$.
\end{enumerate}
Then for all $b_{1},b_{2},\ldots,b_{h}\in\mathbb{Z}_{N}$, 
\[
\mathsf{E}_{x}\left[\prod_{j=1}^{h}g(a_{j}x+b_{j})\right]=\mathsf{E}[g]^h.
\]

Furthermore, if $a_1,a_2,\dots,a_h$ are chosen from $\mathbb{Z}_{N}\setminus \{0\}$ uniformly and independently at random, then Property (\ref{enu:cond4-smoothness}) is satisfied with probability at least $1-\ell^h/(N-1)$. 
\end{lemma} 
\begin{proof} By the Fourier inversion formula, 
\begin{align*}
 & \mathsf{E}_{x}\left[\prod_{j=1}^{h}g(a_{j}x+b_{j})\right]\\
 & =\mathsf{E}_{x}\left[\prod_{j=1}^{h}\left(\sum_{r_{j}\in\mathbb{Z}_{N}}\hat{g}(r_{j})e\left(\frac{r_{j}a_{j}x+r_{j}b_{j}}{N}\right)\right)\right]\\
 & = \sum_{r_{1},r_{2},\ldots,r_{h}\in \mathbb{Z}_N}\prod_{j=1}^{h}e\left(\frac{r_{j}b_{j}}{N}\right)\hat{g}(r_{j})\cdot\mathsf{E}_{x}\left[e\left(\frac{\sum_{j=1}^{h}r_{j}a_{j}x}{N}\right)\right] \\
 & = \sum_{\substack{r_1,r_2\ldots,r_h\in \mathbb{Z}_N,\\ \sum_{j=1}^{h}r_{j}a_{j}=0}} \prod_{j=1}^{h}e\left(\frac{r_{j}b_{j}}{N}\right)\hat{g}(r_{j}).
\end{align*}
Note that $\prod_{j=1}^{h}\hat{g}(0)=\mathsf{E}[g]^h$. 
Consider $(r_{1},r_{2},\ldots,r_{h})\ne(0,\ldots,0)$ where $\sum_{j=1}^{h}r_{j}a_{j}=0$.
Property (\ref{enu:cond4-smoothness}) guarantees that $\hat{g}(r_{j})=0$
for some $j \in [h]$, so $\prod_{j=1}^{h}\hat{g}(r_{j}) = 0$. Hence, if $a_1,a_2,\dots,a_h$ satisfy Property (\ref{enu:cond4-smoothness}), then 
\begin{align*}
\mathsf{E}_{x}\left[\prod_{j=1}^{h}g(a_{j}x+b_{j})\right] = \hat{g}(0)^h = \mathsf{E}[g]^h.
\end{align*}

Next, we show that if $a_1,a_2,\dots,a_h$ are chosen uniformly and independently at random from $\mathbb{Z}_{N}\setminus\{0\}$, then Property (\ref{enu:cond4-smoothness}) is satisfied with probability at least $1-\ell^h/(N-1)$. Indeed, consider a fixed $(r_1,r_2,\dots,r_h) \ne (0,0,\dots,0)$ such that $r_j$ is in the support of $\wh g$ for each $j \in [h]$. There exists $i\in [h]$ such that $r_i\ne 0$. For each fixed choice of $a_j$ for $j \in [h]\setminus \{i\}$, there is a unique choice of $a_i$ such that $\sum_{j=1}^{h}r_ja_j = 0$. Hence, the probability that $\sum_{j=1}^{h}r_ja_j=0$ is at most $1/(N-1)$. By the union bound over the choice of $r_j$ in the support of $\wh g$, we obtain that Property (\ref{enu:cond4-smoothness}) is violated with probability at most $\ell^h/(N-1)$. 
\end{proof} 

Next, for each $\alpha\le 1/2$, we construct a function $g_{\alpha}$ with mean $\alpha$ and prove that $g_{\alpha}$
has the desired properties in Lemma \ref{lemma:smoothness}. We recall that the 3-AP density
of a function $g$ is denoted by $\Lambda(g)=\mathsf{E}_{x,d}[g(x)g(x+d)g(x+2d)]$.

\begin{lemma} \label{lemma:low 3-AP density function}
For $\alpha\le\frac{1}{2}$, define a function $g_{\alpha}:\mathbb{Z}_{N}\to[0,1]$
by $$g_{\alpha}(x)=\alpha - \frac{\alpha \cos(2\pi x/N)}{2} - \frac{\alpha \cos(4\pi x/N)}{2}.$$ Then $g_{\alpha}$ satisfies the following properties. 
\begin{enumerate}[label=(\roman*)]
\item \label{enu:bounded} $\mathsf{E}[g_{\alpha}]=\alpha\le\frac{1}{2}$ and $g_{\alpha}(x)\in[0,2\alpha]$
for all $x\in\mathbb{Z}_{N}$. 
\item \label{enu:expec 3-AP weight} $\Lambda(g_{\alpha})=\mathsf{E}_{x,d}[g_{\alpha}(x)g_{\alpha}(x+d)g_{\alpha}(x+2d)]=(1-\frac{1}{32})\alpha^{3}$,
and $\mathsf{E}_{x}[g_{\alpha}(x)^{3}]\le\frac{3}{2}\alpha^{3}$. 
\item \label{enu:concentration-smoothness}For $h$ a positive integer, if we choose $a_{1},a_{2},\ldots,a_{h}$
uniformly and independently at random from $\mathbb{Z}_{N}\setminus\{0\}$,
then with probability at least $1-5^h/(N-1)$, for all choices of $b_{1},b_{2},\ldots,b_{h}\in \mathbb{Z}_{N}$,
\[
\mathsf{E}_{x}\left[\prod_{j=1}^{h}g_{\alpha}(a_{j}x+b_{j})\right] = \mathsf{E}[g_{\alpha}]^{h}.
\]
\item \label{enu:pairs of different elements} $\mathsf{E}_{x\ne y}[g_{\alpha}(x)g_{\alpha}(y)]\le\alpha^{2}$. 
\item \label{enu:mean square density} $\mathsf{E}_{x}[g_{\alpha}(x)^{2}]=\frac{5}{4}\alpha^{2}$. 
\end{enumerate}
\end{lemma} 

\begin{proof} 
We first observe that $$\wh{g_{\alpha}}(0)=\alpha, \textrm{ and }\wh{g_{\alpha}}(1)=\wh{g_{\alpha}}(N-1)=\wh{g_{\alpha}}(2)=\wh{g_{\alpha}}(N-2)=\alpha / 4,$$ and for all $r\notin \{0,1,2,N-1,N-2\}$, $$\wh{g_{\alpha}}(r)=0.$$ 

Furthermore, it is clear from the definition of $g_{\alpha}$ that for all 
$x\in\mathbb{Z}_{N}$, $g_{\alpha}(x)\in[0,2\alpha]$. Moreover, $\mathsf{E}[g_{\alpha}]=\wh{g_{\alpha}}(0)=\alpha$.
This proves Property \ref{enu:bounded}.

From (\ref{eq:3-APviafourier}), 
\[
\Lambda(g_\alpha)=\frac{1}{N}\sum_{r}\wh{g_{\alpha}}(r)^{2}\wh{g_{\alpha}}(-2r)=\alpha^{3}-2\left(\frac{\alpha}{4}\right)^{3}=\left(1-\frac{1}{32}\right)\alpha^{3},
\]
and 
\begin{align*}
\mathsf{E}_{x}[g_\alpha(x)^{3}] & =\mathsf{E}_{x}\left[\left(\sum_{r}\wh g_\alpha(r)e\left(\frac{rx}{N}\right)\right)^{3}\right]\\
 & =\sum_{r_{1},r_{2},r_{3}\in \mathbb{Z}_N}\wh{g_{\alpha}}(r_{1})\wh{g_{\alpha}}(r_{2})\wh{g_{\alpha}}(r_{3})\mathsf{E}_{x}\left[e\left(\frac{r_{1}x+r_{2}x+r_{3}x}{N}\right)\right]\\
 & =\sum_{\substack{r_{1},r_{2},r_{3}\in \mathbb{Z}_N:\\r_{1}+r_{2}+r_{3}=0}}\wh{g_{\alpha}}(r_{1})\wh{g_{\alpha}}(r_{2}) \wh{g_{\alpha}}(r_{3})\\
 & =\alpha^{3}+6\alpha\frac{\alpha^{2}}{16}-6\frac{\alpha^{3}}{64}<\frac{3}{2}\alpha^{3}.
\end{align*}
This proves Property \ref{enu:expec 3-AP weight}.

Property \ref{enu:concentration-smoothness} follows directly from Lemma \ref{lemma:smoothness} applied to the function $g_\alpha$ and $\ell=5$.

To prove Property \ref{enu:pairs of different elements}, notice
that $\sum_{x}g_\alpha(x)^{2}\ge\frac{1}{N}(\sum_{x}g_\alpha(x))^{2}=\alpha^{2}N$
so 
\[
\mathsf{E}_{x\ne y}[g_\alpha(x)g_\alpha(y)]\le\frac{\alpha^{2}N^{2}-\alpha^{2}N}{N(N-1)}=\alpha^{2}.
\]

Finally, Property \ref{enu:mean square density} follows from Parseval's
identity, 
\[
\mathsf{E}_{x}[g_\alpha(x)^{2}]=\sum_{r}|\wh{g_{\alpha}}(r)|^{2}=\frac{5\alpha^{2}}{4}.
\]

\end{proof}

\section{Lower bound construction for product groups} \label{sec:Product of groups}

In this section, we prove Theorem \ref{th:lower-product}.
For convenience, we recall the theorem statement here. 

\begin{theorem*} \label{th:product of groups}Let $0<\alpha \le 1/4$, $0
<\epsilon \le 20^{-9}$, and $G=\mathbb{Z}_n$ where $n$ is a positive integer such that there exist distinct primes
$m_{1},\ldots,m_{s}$ with $s\le\log_{150}\left(\frac{\epsilon^{-1/4}\alpha^{6}}{8}\right)$ satisfying
\begin{itemize}
\item $n=\prod_{j=1}^{s}m_j$,
\item $\epsilon^{-1/3}/2<m_{1}\le\epsilon^{-1/3}$, and 
\item for $i\ge 2$, $n_{i-1}^{6}<m_{i}<\exp(2^{-1}\cdot64^{-2}\cdot150{}^{i-1}\epsilon^{1/4}n_{i-1})/2$ where $n_i = \prod_{j=1}^{i}m_j$.
\end{itemize}
Then, there exists a function
$f:G\to[0,1]$ with $\mathsf{E}[f]=\alpha$ such that for any $d\in G\setminus\{0\}$,
\[
\mathsf{E}_{x}[f(x)f(x+d)f(x+2d)]\le\alpha^{3}(1-\epsilon).
\]
Furthermore, $\mathsf{E}_{x}[f(x)^3] \le 3\alpha^3/2$ and there exists $\alpha' \in [\alpha,\alpha(1+\epsilon^{1/4})]$ such that $f(x)=\alpha'$ for at least a $3/4$ fraction of $x \in G$.
\end{theorem*} 

We first make a few notation conventions. Note that if $n=\prod_{i=1}^s m_i$ for distinct primes $m_i$, then 
$$G=\mathbb{Z}_n \cong \prod_{i=1}^{s}\mathbb{Z}_{m_{i}}.$$ Each element of $G$ can be represented by an $s$-tuple $(x_{1},x_{2},\ldots,x_{s})$ where $x_{i}\in\mathbb{Z}_{m_{i}}$. Let $Q_i = \prod_{j=1}^i \mathbb{Z}_{m_i}$. We can think of $Q_i$ as a quotient of $G$ by the subgroup $H_i=\{x\in G:x_j=0 \textrm{ for all }j\le i\}$. We identify $\mathbb{Z}_{m_i}$ as the subgroup of $Q_i$ consisting of elements with $x_j=0$ for $j<i$, and we identify the quotient $Q_i/\mathbb{Z}_{m_i}$ with $Q_{i-1}$. We hence use elements of $Q_{i-1}$ to index $\mathbb{Z}_{m_i}$-cosets in $Q_i$. For an element $x\in G$ or $x \in Q_j$ with $j \geq i$, we denote $x_{[i]}=(x_{1},\ldots,x_{i})$. For $j<i$, we say that an element $x$ of $Q_{i}$ is a \textit{lift} of an element $y$ in $Q_{j}$ if $x_{[j]}=y$. In the following discussion, when the level $i$ is clear from context, if not specified otherwise, the 3-APs would refer to 3-APs in $Q_{i}$. 

\subsection{The Construction}

Let $s=\left\lceil \log_{150}\left(\frac{\epsilon^{-1/4}\alpha^{6}}{8}\right)\right\rceil $.
In each level $i$, for $i\in [s]$, we construct a function $f_{i}:Q_{i}\to[0,1]$.
Finally, we let $f=f_{s}:G\to[0,1]$.

We introduce parameters $\mu_{1}=\epsilon^{1/4}$ and $\mu_{i}=150{}^{i-1}\alpha'^{-6}\epsilon^{1/4}$
for $i\ge2$, where $\alpha'=\alpha(1+\frac{1}{m_{1}-1})$.

In the first level, define $f_{1}:Q_{1}\to[0,1]$ by $f_{1}(0)=0$
and $f_{1}(x)=\alpha\left(1+\frac{1}{m_{1}-1}\right)$ for each $x\in Q_{1}\setminus\{0\}$.

For $i\ge2$, we extend $f_{i-1}:Q_{i-1}\to[0,1]$ to a function
$\bar{f}_{i-1}:Q_{i}\to[0,1]$ by setting $\bar{f}_{i-1}(x)=f_{i-1}(x_{[i-1]})$ for
each $x\in Q_{i}$. Let $M_{i-1}$ be any set of $\mu_{i}n_{i-1}$ elements of $Q_{i-1}$ so that $f_{i-1}(x) = \alpha'$ for any $x\in M_{i-1}$. In level $i$, we define $f_i$ to be a random function as follows.

For each $x \in M_{i-1}$, we choose $a_x \in \mathbb{Z}_{m_i}\setminus\{0\}$ and $b_x \in \mathbb{Z}_{m_i}$ uniformly and independently at random. For each $y\in Q_i$ such that $y_{[i-1]}=x$, we define $$f_i(y) = g_{\alpha'}(a_xy_i+b_x) =  \alpha'-\frac{\alpha'\cos(2\pi (a_xy_i+b_x)/m_i)}{2}-\frac{\alpha'\cos(4\pi (a_xy_i+b_x)/m_i)}{2},$$ where $g_{\alpha'}$ is the function with density $\alpha'$ and with low 3-AP density defined earlier in Lemma \ref{lemma:low 3-AP density function}.
Otherwise, for $x\notin M_{i-1}$ and $y\in Q_i$ such that $y_{[i-1]}=x$, we define $$f_i(y)=f_{i-1}(x).$$ 
We refer to this
as the \emph{random modification} in level $i$. This defines (random) $f_{i}:Q_{i}\to[0,1]$. Finally, we let $f=f_s:G \to [0,1]$. We will show that with positive probability, for each level $i$, we can pick $f_i$ such that the function $f$ has the desired properties claimed in Theorem \ref{th:lower-product}. 

\subsection{Proof of Theorem \ref{th:lower-product}}

We first claim that the construction is feasible with the above choice
of parameters. Note that $\mu_1 \geq 1/m_1$, so $f_1(x)=\alpha'$ for all but a $\mu_1$ fraction of elements $x \in Q_1$.  For $i \geq 2$, observe that if $f_i(y) \ne \alpha'$, then we must have $y_{[1]}=0$ or $y_{[j]} \in M_{j}$ for some $j<i$. Thus, the fraction of $y\in Q_{i}$ for which $f_i(y)\ne \alpha'$ is at most $\sum_{j=1}^{i}\mu_{j}$. Since 
\begin{equation}
\sum_{j=1}^{s}\mu_{j}<2\mu_{s}=2\cdot150{}^{s-1}\epsilon^{1/4}\alpha'^{-6}<1/4 \label{eq:frac-alpha'}
\end{equation}
as $s-1\le\log_{150}\left(\frac{\epsilon^{-1/4}\alpha^{6}}{8}\right)<\log_{150}\left(\frac{\epsilon^{-1/4}\alpha'^{6}}{8}\right)$,
it is possible to choose $M_{i}$ for each $i\le s-1$ such that $f_i(x) =\alpha'$ for any $x\in M_{i}$. 

We next prove that the function $f_i$ has density $\alpha$ and $f_i$ maps $Q_i$ to $[0,1]$. This is true for $i=1$. Assume that $f_{i-1}$ has density $\alpha$ and takes values in $[0,1]$, we show that $f_i$ also has these properties. For $x\in Q_i$ such that $x_{[i-1]} \notin M_{i-1}$, $f_i(x)=f_{i-1}(x_{[i-1]}) \in [0,1]$. If $x_{[i-1]} \in M_{i-1}$ then $f_{i-1}(x_{[i-1]}) = \alpha' \le 1/2$. Hence, if $x_{[i-1]}\in M_{i-1}$ then $f_i(x)=g_{\alpha'}(ax_i+b)$ for some $a\in \mathbb{Z}_{m_i}\setminus\{0\},b\in \mathbb{Z}_{m_i}$. Since $g_{\alpha'}$ also has density $\alpha'$ and takes values in $[0,1]$, we have $f_i(x) = g_{\alpha'}(ax_i+b) \in [0,1]$ and the density of $f_i$ over the $\mathbb{Z}_{m_i}$-coset $x_{[i-1]}$ is $f_{i-1}(x_{[i-1]})=\alpha'$. Hence, the density of $f_i$ is the same as the density of $f_{i-1}$, and $f_i$ takes values in $[0,1]$. By induction, the density of $f_i$ is $\alpha$ and the values of $f_{i}$ are in $[0,1]$ for all $i \in [s]$. 

We denote by $\mathbb{E}_{f_{i}}$ the expectation over the randomness
of $f_{i}$ (the local modifications in level $i$), conditioned on
a fixed choice of $f_{i-1}$. Furthermore, all of the probability
we consider will be conditioned on this fixed choice of $f_{i-1}$,
hence in level $i$ we only consider the randomness of the random modification
in level $i$. 

The random modification in level $i$ has the following key property.
For any $x=(x_{1},\ldots,x_{i})\in Q_{i}$ such that $x_{[i-1]}=(x_{1},\ldots,x_{i-1})\in M_{i-1}$
and $d\in Q_{i}\setminus\{0\}$ such that $d_{[i-1]}=0\in Q_{i-1}$, we have 
\begin{align}
\mathbb{E}_{f_{i}}[f_{i}(x)f_{i}(x+d)f_{i}(x+2d)]&=\mathbb{E}_{a\in \mathbb{Z}_{m_i}\setminus\{0\},b\in \mathbb{Z}_{m_i}}[g_{\alpha'}(ax_{i}+b)g_{\alpha'}(ax_{i}+ad_{i}+b)g_{\alpha'}(ax_{i}+2ad_{i}+b)] \nonumber \\ 
&\le\Lambda(g_{\alpha'}) \nonumber\\
&=\frac{31}{32}\alpha'^{3}.\label{eq:single 3-AP small weight}
\end{align}
This is since when $a$ is chosen uniformly at random from $\mathbb{Z}_{m_i}\setminus\{0\}$ and $b$ is chosen uniformly at random from $\mathbb{Z}_{m_i}$, then for any fixed $x_i$ and nonzero $d_i$, $(ax_i+b,ax_i+ad_i+b,ax_i+2ad_i+b)$ is distributed uniformly among all 3-APs with nonzero common difference in $\mathbb{Z}_{m_i}$. 

We now proceed to prove that there exists a choice of the modification
in each level so that for any $d\in G\setminus\{0\}$,
\[
\mathsf{E}_{x}[f(x)f(x+d)f(x+2d)]\le\alpha^{3}(1-\epsilon).
\]
The main idea is to maintain by induction that for any $i\in [s]$,
we can choose $f_{i}$ which is a random modification of $f_{i-1}$ so that for any $d\in Q_{i}\setminus\{0\}$,
\[
\mathsf{E}_{x\in Q_{i}}[f_{i}(x)f_{i}(x+d)f_{i}(x+2d)]\le\alpha^{3}(1-\epsilon).
\]
For all $d$ such that $d_{[i-1]}=0$, the above property follows from observation (\ref{eq:single 3-AP small weight}) and concentration inequalities. If $d_{[i-1]}\ne0\in Q_{i-1}$,
we have  $\mathsf{E}_{x\in Q_{i-1}}[f_{i-1}(x)f_{i-1}(x+d_{[i-1]})f_{i-1}(x+2d_{[i-1]})]$
is small by the induction hypothesis. We guarantee that with large probability, $\mathsf{E}_{x\in Q_{i}}[f_{i}(x)f_{i}(x+d)f_{i}(x+2d)]=\mathsf{E}_{x\in Q_{i-1}}[f_{i-1}(x)f_{i-1}(x+d_{[i-1]})f_{i-1}(x+2d_{[i-1]})]$ for all $d$ such that $d_{[i-1]}\ne 0$, so we also have $\mathsf{E}_{x\in Q_{i}}[f_{i}(x)f_{i}(x+d)f_{i}(x+2d)]$ is small. Combining these two cases, we obtain a modification $f_i$ of $f_{i-1}$ whose density of 3-APs with common difference $d$ is small for all nonzero $d\in Q_i$. 

We now give the proof of Theorem \ref{th:lower-product}.

\begin{proof}[Proof
of Theorem \ref{th:lower-product}] It is easy to see that
\begin{align*}
\mathsf{E}_{x\in Q_{1}}[f_{1}(x)^{3}] & =\alpha^{3}\frac{\left(1+\frac{1}{m_{1}-1}\right)^{3}(m_{1}-1)}{m_{1}}\\
 & <\alpha^{3}\left(1+\frac{3}{m_{1}}\right)\le\alpha^{3}(1+6\epsilon^{1/3})\le\alpha^{3}(1+2\epsilon^{1/4})\\
 & <\alpha'^{3}(1+2\mu_{1}),
\end{align*}
for $\epsilon\le20^{-9}$. Inductively, if 
\[
\mathsf{E}_{x\in Q_{i-1}}[f_{i-1}(x)^{3}]\le\alpha'^{3}(1+2\mu_{i-1}),
\]
then 
\begin{equation}\label{forremarkmeancubedensitybound}
\mathsf{E}_{x\in Q_{i}}[f_{i}(x)^{3}]\le\alpha'^{3}(1+2\mu_{i-1})+\frac{1}{2}\mu_{i}\alpha'^{3}<\alpha'^{3}(1+2\mu_{i}),
\end{equation}
where the first inequality is by \ref{enu:expec 3-AP weight} in Lemma \ref{lemma:low 3-AP density function} as we apply the local modification to a $\mu_i$ fraction of the $\mathbb{Z}_{m_i}$-cosets, getting at most a $\frac{1}{2}\alpha'^3$ increment in the mean cube density over each of them, and the second inequality follows from our choice of parameters $\mu_{i}\ge150\mu_{i-1}$ for all
$i\ge2$.

Let ${\mathcal P}(i)$ be the property that for all $d\in Q_i \setminus \{0\}$, 
\[\mathsf{E}_{x\in Q_{i}}[f_{i}(x)f_{i}(x+d)f_{i}(x+2d)] \leq \alpha^3(1-\epsilon).\]

We prove by induction that in level $i$, the modifications can be chosen so that ${\mathcal P}(i)$ holds.

Consider the base case $i=1$. Recall that $\epsilon^{-1/3}/2\le m_{1}\le\epsilon^{-1/3}$.
For any $d\in Q_{1}\setminus\{0\}$, 
\begin{align*}
\mathsf{E}_{x\in Q_{1}}[f_{1}(x)f_{1}(x+d)f_{1}(x+2d)] & =\alpha^{3}\frac{\left(1+\frac{1}{m_{1}-1}\right)^{3}(m_{1}-3)}{m_{1}}\\
 & =\alpha^{3}\frac{m_{1}^{2}(m_{1}-3)}{(m_{1}-1)^{3}}=\alpha^{3}\left(1-\frac{3m_{1}-1}{(m_{1}-1)^{3}}\right)\\
 & \le\alpha^{3}\left(1-\frac{1}{m_{1}^{2}}\right)\le\alpha^{3}(1-\epsilon).
\end{align*}

This establishes ${\mathcal P}(1)$. Next, we continue with the inductive
step. Assume that ${\mathcal P}(i-1)$ holds. We prove that we can choose
the modification in level $i$ so that ${\mathcal P}(i)$ also holds. This follows from the following two claims.

\begin{claim} \label{claim:1}
With probability larger than $1/2$, conditioned on a fixed choice of $f_{i-1}$ satisfying $\mathcal{P}(i-1)$, for all $d\in Q_{i}\setminus\{0\}$
with $d_{[i-1]}=0$,
\[
\mathsf{E}_{x\in Q_{i}}[f_{i}(x)f_{i}(x+d)f_{i}(x+2d)]\le\alpha^{3}(1-\epsilon).
\]
\end{claim}
\begin{claim}\label{claim:2}
With probability larger than $1/2$, conditioned on a fixed choice of $f_{i-1}$ satisfying $\mathcal{P}(i-1)$, for all $d\in Q_{i}\setminus\{0\}$
with $d_{[i-1]}\ne 0$,
\[
\mathsf{E}_{x\in Q_{i}}[f_{i}(x)f_{i}(x+d)f_{i}(x+2d)]=\mathsf{E}_{x\in Q_{i-1}}[f_{i-1}(x)f_{i-1}(x+d')f_{i-1}(x+2d')].
\]
\end{claim}

Combining Claims \ref{claim:1} and \ref{claim:2}, by the union bound, the modification in level $i$ fails to satisfy ${\mathcal P}(i)$ with probability strictly less than $1$. Thus we can choose a modification satisfying ${\mathcal P}(i)$ in level $i$. This completes the induction. Thus, there exists $f=f_s$ which satisfies $\mathcal{P}(s)$, so for any nonzero $d$ in $G$, 
\[
\mathsf{E}_{x\in G}[f(x)f(x+d)f(x+2d)]=\mathsf{E}_{x\in G}[f_{s}(x)f_{s}(x+d)f_{s}(x+2d)]\le\alpha^{3}(1-\epsilon).
\] This completes the proof of Theorem \ref{th:lower-product}. 

Now we turn to the proofs of Claims \ref{claim:1} and \ref{claim:2}.

\begin{proof}[Proof of Claim \ref{claim:1}]
Let $d\in Q_i\setminus\{0\}$ such that $d_{[i-1]}=0$. By (\ref{eq:single 3-AP small weight}), for any $x\in Q_{i}$ with $x_{[i-1]}\in M_{i-1}$, 
\[
\mathbb{E}_{f_{i}}[f_{i}(x)f_{i}(x+d)f_{i}(x+2d)]\le\frac{31}{32}\alpha'^{3}.
\]
Hence, for $y\in M_{i-1}$,
\[
\mathbb{E}_{f_{i}}\mathsf{E}_{x\in Q_{i},x_{[i-1]}=y}[f_{i}(x)f_{i}(x+d)f_{i}(x+2d)]\le\frac{31}{32}\alpha'^{3}.
\]
Note that the random variables $\mathsf{E}_{x\in Q_{i},x_{[i-1]}=y}[f_{i}(x)f_{i}(x+d)f_{i}(x+2d)]$, for $y\in M_{i-1}$,
are independent (under the randomness of the
modification in level $i$, conditioned on a fixed choice of $f_{i-1}$). Thus the probability that 
\[
\mathsf{E}_{y\in M_{i-1}}\mathsf{E}_{x\in Q_{i},x_{[i-1]}=y}[f_{i}(x)f_{i}(x+d)f_{i}(x+2d)]\ge\frac{63}{64}\alpha'^{3}
\]
is at most $\exp(-2^{-1}\cdot64^{-2}\mu_{i}n_{i-1}\alpha'^{6})$ by
Hoeffding's inequality. 

For $d\in Q_i$ such that $d_{[i-1]}=0$, we have 
\begin{align*}
&\mathsf{E}_{x\in Q_{i}}[f_{i}(x)f_{i}(x+d)f_{i}(x+2d)] \\
&= \mathsf{E}_{y\in Q_{i-1}}[f_{i-1}(y)^3] + \frac{|M_{i-1}|}{|Q_{i-1}|}\left(\mathsf{E}_{y\in M_{i-1}}\mathsf{E}_{x\in Q_{i},x_{[i-1]}=y}[f_{i}(x)f_{i}(x+d)f_{i}(x+2d)] - \mathsf{E}_{y\in M_{i-1}} [f_{i-1}(y)^3] \right)\\
&\le \alpha'^3(1+2\mu_{i-1}) + \mu_i \cdot \left(\mathsf{E}_{y\in M_{i-1}}\mathsf{E}_{x\in Q_{i},x_{[i-1]}=y}[f_{i}(x)f_{i}(x+d)f_{i}(x+2d)] - \alpha'^3 \right)
\end{align*} where the first equality follows from $f_i(x)=f_i(x+d)=f_i(x+2d)=f_{i-1}(y)$ if $d_{[i-1]}=0$ and $x_{[i-1]}=y\notin M_{i-1}$, and the inequality follows from (\ref{forremarkmeancubedensitybound}) and $f_{i-1}(y)=\alpha'$ for $y\in M_{i-1}$. Thus, if \[
\mathsf{E}_{y\in M_{i-1}}\mathsf{E}_{x\in Q_{i},x_{[i-1]}=y}[f_{i}(x)f_{i}(x+d)f_{i}(x+2d)]\le\frac{63}{64}\alpha'^{3},
\] then $$\mathsf{E}_{x\in Q_{i}}[f_{i}(x)f_{i}(x+d)f_{i}(x+2d)] \le \alpha'^3(1+2\mu_{i-1})-\mu_i\alpha'^3/64.$$

Since 
\[
\alpha^{3}(1-\epsilon)\ge\alpha'^{3}(1-\epsilon^{1/4})>\alpha'^{3}(1+2\mu_{i-1})-\mu_{i}\alpha'^{3}/64,
\]
by the union bound, the probability that there exists $d\in Q_{i}\setminus\{0\}$
with $d_{[i-1]}=0$ and 
\[
\mathsf{E}_{x\in Q_{i}}[f_{i}(x)f_{i}(x+d)f_{i}(x+2d)]\ge\alpha^{3}(1-\epsilon)
\]
is at most 
\[
m_{i}\exp(-2^{-1}\cdot64{}^{-2}\mu_{i}n_{i-1}\alpha'^{6})<1/2,\] 
where we used the upper bound on $m_i$ in the theorem statement. 
\end{proof}

\begin{proof}[Proof of Claim \ref{claim:2}]

Recall that for a $\mathbb{Z}_{m_i}$-coset representing by $w\in Q_{i-1}$, $f_i$ is either a constant function on $w$ if $w\notin M_{i-1}$, or otherwise $f_i(x)=g_{\alpha'}(a_wx_i+b_w)$ where $$g_{\alpha'}(x)=\alpha'-\frac{\alpha'\cos(2x/m_i)}{2}-\frac{\alpha'\cos(4x/m_i)}{2}$$ as defined in Lemma \ref{lemma:low 3-AP density function} and $a_w\in \mathbb{Z}_{m_i}\setminus\{0\}$ and $b_w\in \mathbb{Z}_{m_i}$ are chosen uniformly and independently for each $w\in M_{i-1}$. For each 3-AP $(w,w+d',w+2d')$ with common difference $d'\in Q_{i-1}\setminus\{0\}$, and for any lift $d$ of $d'$, we have  
\begin{align}
&\mathsf{E}_{x\in Q_{i},x_{[i-1]}=w}[f_{i}(x)f_{i}(x+d)f_{i}(x+2d)]\\
&=\mathsf{E}_{y\in\mathbb{Z}_{m_{i}}}[g_{(1)}(a_{1}y+b_{1})g_{(2)}(a_{2}y+a_2d_i+b_{2})g_{(3)}(a_{3}y+2a_3d_i+b_{3})]\nonumber\\
&=\mathsf{E}_{y\in\mathbb{Z}_{m_{i}}}[g_{(1)}(a_{1}y+c_{1})g_{(2)}(a_{2}y+c_{2})g_{(3)}(a_{3}y+c_{3})]\label{eq:count-3-Ap-d},
\end{align}
where $g_{(j)}:\mathbb{Z}_{m_i}\to [0,1]$ can be either the function $g_{\alpha'}$ or a constant function, $a_{j}\in \mathbb{Z}_{m_i}\setminus\{0\},b_j\in \mathbb{Z}_{m_i}$ are chosen uniformly and independently at random, and $c_1=b_1$, $c_2 = a_2d_i+b_2$, $c_3 =2a_3d_i+b_3$. Note that if we fix the modification (i.e., fixing each $a_{j}$
and $b_{j}$), changing $d$ to a different lift of $d'$ would only
change $c_{j}$ in equation (\ref{eq:count-3-Ap-d}), and would
not change the coefficients of $y$ in $g_{(1)},g_{(2)},g_{(3)}$ in the last line of equation (\ref{eq:count-3-Ap-d}). Let $J\subseteq [3]$ be the set of indices such that $g_{(j)}=g_{\alpha'}$. By Lemma \ref{lemma:smoothness} applied to the function $g_{\alpha'}$ and $h=|J|\le3 $, with probability at least $1-125/(m_i-1)$, $$\mathsf{E}_{y\in\mathbb{Z}_{m_{i}}}[\prod_{j\in J}g_{(j)}(a_{j}y+u_j)] =\prod_{j\in J}\mathsf{E}_{y\in \mathbb{Z}_{m_i}}[g_{(j)}(y)]$$ for all $u_j\in \mathbb{Z}_{m_i}$. Since $g_{(j)}$ is a constant function for $j\notin J$, we obtain that with probability at least $1-125/(m_i-1)$,
\begin{align*}
\mathsf{E}_{y\in\mathbb{Z}_{m_{i}}}[g_{(1)}(a_{1}y+c_{1})g_{(2)}(a_{2}y+c_{2})g_{(3)}(a_{3}y+c_{3})] &= \mathsf{E}_{y\in\mathbb{Z}_{m_{i}}}[g_{(1)}(y)]\mathsf{E}_{y\in\mathbb{Z}_{m_{i}}}[g_{(2)}(y)]\mathsf{E}_{y\in\mathbb{Z}_{m_{i}}}[g_{(3)}(y)] \\
&=f_{i-1}(w)f_{i-1}(w+d')f_{i-1}(w+2d'). 
\end{align*}

Thus, by the union bound, with probability at least $1-125n_{i-1}^2/(m_i-1)$, for every 3-AP $(w,w+d',w+2d')$ in $Q_{i-1}$ with nonzero common difference $d'$, and for all lifts $d$ of $d'$ in $Q_i$, $$\mathsf{E}_{x\in Q_{i},x_{[i-1]}=w}[f_{i}(x)f_{i}(x+d)f_{i}(x+2d)]=f_{i-1}(w)f_{i-1}(w+d')f_{i-1}(w+2d').$$ For $i\ge2$, $m_i \ge n_{i-1}^6$, and $m_{i}\ge m_{2}\ge \epsilon^{-2}/64 > 10^6$,
so $125n_{i-1}^2/(m_i-1)<1/2$. Hence with probability larger than $1/2$, for all $d\in Q_i$ such that $d_{[i-1]}=d'\ne 0$, $$\mathsf{E}_{x\in Q_{i}}[f_{i}(x)f_{i}(x+d)f_{i}(x+2d)]=\mathsf{E}_{x\in Q_{i-1}}f_{i-1}(x)f_{i-1}(x+d')f_{i-1}(x+2d'). \qedhere$$
\end{proof}

Thus, assuming that $\mathcal{P}(i-1)$ holds, we can choose the modification in level $i$ so that $\mathcal{P}(i)$ holds. By induction, we can find a function $f_s$ which satisfies $\mathcal{P}(s)$. Notice that $f_s(x)=\alpha'$ for at least a $3/4$ fraction of $x\in G$ by (\ref{eq:frac-alpha'}), and $\mathbb{E}_{x}[f_s(x)^3] \le 3\alpha^3/2$ by (\ref{forremarkmeancubedensitybound}) with $i=s$. The function $f=f_s$ then satisfies the conclusion of Theorem \ref{th:lower-product}. 
\end{proof}

\section{Lower bound construction for intervals} \label{sec:Intervals}

In this section we prove Theorem \ref{th:Interval-1}, restated below for convenience.

\begin{theorem*} \label{th:Interval} There are positive absolute constants
$c,\alpha_{0}$ such that the following holds. If $0\le\alpha\le\alpha_{0}$, $0 < \epsilon\le \alpha^{7}$, and $N\le \tower(c\log(1/\epsilon))$, then there is a function $f:[N]\to[0,1]$ with $\mathsf{E}[f]=\alpha$ such that for any $0<d<N/2$, 
\[
\mathsf{E}_{x\in[N-2d]}[f(x)f(x+d)f(x+2d)]\le\alpha^{3}(1-\epsilon).
\] 
\end{theorem*} 

By Appendix \ref{appendix:lower}, in order to prove Theorem \ref{th:Interval-1}, we can (and will) assume that $N\ge \epsilon^{-15}$. 

Before proving Theorem \ref{th:Interval-1}, we first need an auxiliary construction of a set with relatively low $3$-AP density given its density. Recall from the introduction that $N(\alpha)$ is the least positive integer such that if $N \geq N(\alpha)$ and $A \subset [N]$ with $|A| \geq \alpha N$, then $A$ contains a $3$-AP.

\begin{lemma} For $\alpha>0$ sufficiently small, there is a subset $T \subset \mathbb{Z}_n$ with $|T| \geq \alpha n$ and with $3$-AP density at most $\max\left(\frac{1}{n},\frac{2\alpha}{N(6\alpha)}\right)$.
\end{lemma}
\begin{proof}
Let $N=N(6\alpha)-1$, so there is  $A \subset [N]$ with $|A|=\lceil 6 \alpha N \rceil$ which has no nontrivial $3$-AP. 

First assume $n \leq 4N$. Partition $[N]$ into at most $2N/n+1 \leq 6N/n$ intervals of length at most $\lceil n/2 \rceil$. The set $A$ contains at least $|A|/(6N/n)\geq \alpha n$ elements in one of these intervals. Viewed as a subset of $\mathbb{Z}_n$, we have a subset of $\mathbb{Z}_n$ with density at least $\alpha$ and with no nontrivial $3$-AP, and hence $3$-AP density at most $1/n$. 

So we may assume $n > 4N$. Integers $x,y,z$ form an {\it approximate $3$-AP} if $|2z-x-y| \leq 1$. Let $S:=\{2a:a \in A\}$, so $S$ has no approximate $3$-AP. Let $t=\lfloor \frac{n}{4N}\rfloor$. Consider the set $I_i:=\{(i-1)t+1,(i-1)t+2,\ldots,(i-1)t+t\}$ of $t$ consecutive integers. Let $T$ be the union of the sets $I_i$ with $i \in S$. The set $T$ has size $|T|=|A|t \geq \alpha n$. Also, every element of $T$ is a positive integer at most $(2N-1)t+t \leq n/2$. So if $x,y,z\in T$ are such that $(x,y,z) \pmod n$ form a $3$-AP in $\mathbb{Z}_n$, then $(x,y,z)$ is also a $3$-AP of integers. 
Since $S$ has no approximate $3$-AP, it follows that the only $3$-APs in $T$ are those where the three terms are in the same interval $I_i$. In each interval $I_i$, which has size $t$, the number of $3$-APs (with any integer difference allowed) is $t+2\lfloor \frac{t^2-1}{4}\rfloor$. There are $|A|$ intervals $I_i$ whose union is $T$. The number of $3$-APs in $\mathbb{Z}_n$ is $n^2$. Hence, the $3$-AP density of $T$ as a subset of $\mathbb{Z}_n$ is $\left(t+2\lfloor \frac{t^2-1}{4}\rfloor\right) |A|/n^2 \leq \frac{2\alpha}{N(6\alpha)}$. 
\end{proof}

The Behrend construction \cite{B} implies that on $N(\alpha)$ implies that if $\alpha>0$ is sufficiently small, then $N(6\alpha) \geq 2^{\frac{1}{9}(\log_2 1/\alpha)^2}$. Together with the previous lemma, we have the following immediate corollary. 

\begin{lemma} \label{lemma:Behrend-construction} If $\alpha>0$ is sufficiently small, then for any positive integer $n$, there is a subset of $\mathbb{Z}_{n}$ with density at least $\alpha$ and 3-AP density
at most $\max\left(\frac{1}{n},2^{-\frac{1}{9}(\log_2 1/\alpha)^2}\right)$.\end{lemma} 

\subsection{The construction and proof of Theorem~\ref{th:Interval-1}}\label{subsection:interval-construction}

We next construct a function $f:[N]\to[0,1]$ with $\mathsf{E}_{x}[f(x)]=\alpha$
such that for any $0<d<N/2$, $\mathsf{E}_{x\in[N-2d]}[f(x)f(x+d)f(x+2d)]\le\alpha^{3}(1-\epsilon)$.
The construction is done in three steps. 

\noindent \textbf{Step 1:} Choose $\beta$ so that $\beta\le\epsilon^{2}$
and $N'=N(1-\beta)$ has a divisor $q$ such that $N^{1/5}<q<\sqrt{\beta\alpha^{3}(1-\epsilon)N}$, $N'/q$ is prime, and $q$ satisfies the condition in Theorem \ref{th:lower-product} with parameters $\alpha_{3.4} = \alpha$ and $\epsilon_{3.4} = 4\epsilon$. Here, for clarity, we include the theorem index in the subscript of the parameters in the theorem. Note that if $q$ satisfies the condition in Theorem \ref{th:lower-product} with parameters $\alpha_{3.4}=\alpha$ and $\epsilon_{3.4}=4\epsilon$, then for any $\alpha_1 \in [\alpha,1/4]$, $q$ also satisfies the condition in Theorem \ref{th:lower-product} with parameters $\alpha_{3.4}=\alpha_1$ and $\epsilon_{3.4} = 4\epsilon$. The existence of $\beta,N',q$ is guaranteed by Lemma \ref{lem:approx2}, which is deferred to the end of the section. Let $$\alpha' \coloneqq (1-\beta)^{-1}\alpha.$$ 

\noindent \textbf{Step 2:} Since $\alpha' \in [\alpha,1/4]$, we can apply Theorem \ref{th:lower-product} with $G=\mathbb{Z}_q$, $\alpha_{4.4}=\alpha'$, $\epsilon_{4.4} = 4\epsilon$. We obtain $g:\mathbb{Z}_{q}\to[0,1]$ with density $\alpha'$, mean cube density at most
$\frac{3}{2}\alpha'^{3}$, such that for each $d\in\mathbb{Z}_{q}\setminus\{0\}$,
\[
\mathsf{E}_{x}[g(x)g(x+d)g(x+2d)]\le\alpha'^{3}(1-4\epsilon)\le\alpha^{3}(1-3\epsilon),
\]
and there exists $\alpha^{*}\in[\alpha',\alpha'(1+\epsilon^{1/4})]$ such that $|\{x\in\mathbb{Z}_{q}:g(x)=\alpha^{*}\}|\ge3q/4$.
For an integer $x$, denote $\bar{x}_{q}=x\mod q\in\mathbb{Z}_{q}$.
Define $f_{2}:[N]\to[0,1]$ by $f_{2}(x)=g(\bar{x}_{q})$ for $x\in[N']$
and $f_{2}(x)=0$ for $x>N'$.

\noindent \textbf{Step 3:} Let $n=\frac{N'}{q}$, which is prime.
Apply Lemma \ref{lemma:Behrend-construction} to find $X\subset\mathbb{Z}_{n}$ with density at least $\alpha^{*}$ and 3-AP density at most $\max\left(1/n,2^{-\log_2(1/\alpha^*)^2/9}\right)$.
We use $\xi$ to denote the characteristic function of $X$ scaled by $\alpha^* n/|X|$, so $\xi(x)=\alpha^* n X(x)/|X|$. Then $\mathsf{E}_{x\in \mathbb{Z}_n}[\xi(x)]=\alpha^*$ and since $|X|/n\ge \alpha^*$, $$\mathsf{E}_{x,d\in \mathbb{Z}_n}[\xi(x)\xi(x+d)\xi(x+2d)] \le \mathsf{E}_{x,d\in \mathbb{Z}_n}[X(x)X(x+d)X(x+2d)] \le \max\left(1/n,2^{-\log_2(1/\alpha^*)^2/9}\right).$$ For each
$t\in\mathbb{Z}_{q}$, we define $P_t = \{x\in[N'],\bar{x}_{q}=t\}$, which forms an arithmetic
progression of length $n$. Let $x_{1}<x_{2}<\ldots<x_{n}$ be the elements
of $P_t$ in increasing order. We define a bijection $\phi_{t}:P_t\to\mathbb{Z}_{n}$
such that $\phi_{t}(x_{i})= i \mod n$. Observe that if $(x,y,z)$ is a 3-AP
in $P_t$, then $(\phi(x),\phi(y),\phi(z))$
is a 3-AP in $\mathbb{Z}_{n}$. For each $t\in\mathbb{Z}_{q}$, we
choose independently and uniformly at random $a_{t}\in\mathbb{Z}_{n}\setminus\{0\},b_{t}\in\mathbb{Z}_{n}$,
with independent choices for different $t$. Define $f_{3}$ by $f_{3}(x)= \xi(a_{t}\phi_{t}(x)+b_{t})$
for $x$ such that $\bar{x}_{q}=t$ and $g(t)=\alpha^{*}$, and $f_{3}(x)=f_{2}(x)$
otherwise.

We let $f=f_{3}$. It is easy to see that $\mathsf{E}_{x}[f_{3}(x)]=\alpha$. We
now prove that there exists a choice of randomness (in Step 3) such
that for each positive integer $d<N/2$,
\[
\mathsf{E}_{x\in[N-2d]}[f(x)f(x+d)f(x+2d)]\le\alpha^{3}(1-\epsilon).
\]

\begin{proof}[Proof of Theorem \ref{th:Interval-1}] 
We reuse the notations from the description of the construction. For a 3-AP $(x,x+d,x+2d)$ in $[N]$, we refer to $f_3(x)f_3(x+d)f_3(x+2d)$ as its {\it weight}. For any common difference $d\ge\frac{N'}{2}$, all the 3-APs $(x,x+d,x+2d)$ in $[N]$ have zero weight since $x+2d>N'$. Hence the density of 3-APs with common difference $d$ of $f_3$ is $0$. For any common difference $d\ge \frac{N'-\beta N\alpha^{3}(1-\epsilon)}{2}$, let $t=N'-2d$. Then $t\le\beta N\alpha^{3}(1-\epsilon)$. The number of 
3-APs with common difference $d$ in $[N']$ is at most $t$, and hence the  number of 3-APs in $[N]$ with
nonzero weight is at most $t$. The number of 3-APs with common difference $d$ in $[N]$ is $N-2d=\beta N+t$, so the density of 3-APs with common difference
$d$ in $[N]$ is at most $\frac{t}{\beta N+t}<\alpha^{3}(1-\epsilon)$ since
$t\le\frac{\beta N\alpha^{3}(1-\epsilon)}{1-\alpha^{3}(1-\epsilon)}$.

For $d$ such that $0<d<\frac{N'-\beta N\alpha^{3}(1-\epsilon)}{2}$, the number of 3-APs of common difference
$d$ in $[N']$ is at least $\beta N\alpha^{3}(1-\epsilon)$. Partition the 3-APs with common difference $d$ in
$[N']$ into different classes according to the congruence class modulo
$q$ of the 3-AP (so the class a 3-AP belongs to is determined by the congruence class modulo $q$ of the first element of the 3-AP). Since $\beta N\alpha^{3}(1-\epsilon)>q^{2}$, all classes of
3-APs modulo $q$ with common difference $\bar{d}_{q}$ appear, each
class with at least $q$ elements, and any two classes differ in size by at most $1$. Hence, 
\[
\left|\mathsf{E}_{x\in[N'-2d]}[f_{2}(x)f_{2}(x+d)f_{2}(x+2d)]-\mathsf{E}_{y\in\mathbb{Z}_{q}}[g(y)g(y+\bar{d}_{q})g(y+2\bar{d}_{q})]\right|\le\frac{1}{q}.
\]
By the construction, if $\bar{d}_{q}\ne0$ then 
\[
\mathsf{E}_{y\in\mathbb{Z}_{q}}[g(y)g(y+\bar{d}_{q})g(y+2\bar{d}_{q})]\le\alpha^{3}(1-3\epsilon),
\]
and if $\bar{d}_{q}=0$ then 
\[
\mathsf{E}_{y\in\mathbb{Z}_{q}}[g(y)g(y+\bar{d}_{q})g(y+2\bar{d}_{q})]\le\frac{3}{2}\alpha'^{3}.
\]
Thus, for $d$ nonzero modulo $q$, 
\[
\mathsf{E}_{x\in[N'-2d]}[f_{2}(x)f_{2}(x+d)f_{2}(x+2d)]\le\alpha^{3}(1-3\epsilon)+\frac{1}{q}<\alpha^{3}(1-2\epsilon),
\]
and for $d$ divisible by $q$, 
\begin{equation} \label{eq:f2}
\mathsf{E}_{x\in[N'-2d]}[f_{2}(x)f_{2}(x+d)f_{2}(x+2d)]\le\frac{3}{2}\alpha'^{3} + \frac{1}{q}.
\end{equation}

In the third step, suppose $d$ is nonzero and divisible by $q$ and let
$t\in\mathbb{Z}_{q}$ with $g(t)=\alpha^{*}$. For $x\in[N'-2d]$
with $\bar{x}_{q}=t$, one has $(f_{3}(x),f_{3}(x+d),f_{3}(x+2d))=(\xi(a_{t}\phi_{t}(x)+b_{t}), \xi(a_{t}\phi_{t}(x+d)+b_{t}), \xi(a_{t}\phi_{t}(x+2d)+b_{t}))$.
Recall that $a_t$ is uniformly distributed over $\mathbb{Z}_n\setminus\{0\}$ and $b_t$ is uniformly distributed over $\mathbb{Z}_n$, so $(a_{t}\phi_{t}(x)+b_{t},a_{t}\phi_{t}(x+d)+b_{t},a_{t}\phi_{t}(x+2d)+b_{t})$
is uniformly distributed over the 3-APs in $\mathbb{Z}_{n}$ with
nonzero common difference. Thus, 
\begin{equation*} \label{eq:dens-3AP-lev3}
\mathbb{E}_{f_{3}}\mathsf{E}_{\substack{\bar{x}_{q}=t,x\in[N'-2d]}
}[f_{3}(x)f_{3}(x+d)f_{3}(x+2d)] = \lambda(\xi) \le \Lambda(\xi)\le \max\left(\frac{1}{n},2^{-\log_2 (1/\alpha^*)^2/9}\right) \le \frac{\alpha^{*}{}^{3}}{10},
\end{equation*}
where $\lambda(\xi)$ is the density of 3-APs with nonzero common difference of $\xi$, $\Lambda(\xi)$ is the density of 3-APs of $\xi $, and we have used that $n \ge \sqrt{N}/2 \ge \epsilon^{-15/2}/2 \ge \alpha^{-10}$ and $\alpha\le \alpha_0$ is sufficiently small. Thus, for each $t\in \mathbb{Z}_{q}$ such that $g(t)=\alpha^{*}$, 
\begin{equation}\label{eq:t-decrease}\mathbb{E}_{f_{3}}\mathsf{E}_{\substack{\bar{x}_{q}=t,x\in[N'-2d]}
}[f_{3}(x)f_{3}(x+d)f_{3}(x+2d)] \le \mathsf{E}_{\substack{\bar{x}_{q}=t,x\in[N'-2d]}
}[f_{2}(x)f_{2}(x+d)f_{2}(x+2d)] - \frac{9\alpha^{*}{}^{3}}{10}.
\end{equation} For each $t\in \mathbb{Z}_{q}$ such that $g(t)\ne \alpha^{*}$, 
\begin{equation}\label{eq:t-same} \mathbb{E}_{f_{3}}\mathsf{E}_{\substack{\bar{x}_{q}=t,x\in[N'-2d]}
}[f_{3}(x)f_{3}(x+d)f_{3}(x+2d)] = \mathsf{E}_{\substack{\bar{x}_{q}=t,x\in[N'-2d]}
}[f_{2}(x)f_{2}(x+d)f_{2}(x+2d)].
\end{equation}
Hence,
\[
\mathbb{E}_{f_{3}}\mathsf{E}_{x\in[N'-2d]}[f_{3}(x)f_{3}(x+d)f_{3}(x+2d)] 
\le\frac{3\alpha^{*}{}^{3}}{2}+\frac{1}{q}-\frac{3}{4}\cdot\frac{9\alpha^{*}{}^{3}}{10}\le \frac{5\alpha^{*}{}^{3}}{6},
\] where in the first inequality we used (\ref{eq:f2}), (\ref{eq:t-decrease}) and (\ref{eq:t-same}) together with the fact that $g(t)=\alpha^*$ for at least a $3/4$ fraction of $t \in \mathbb{Z}_q$, and in the second inequality we used that $q>N^{1/5}\ge \epsilon^{-3} \ge \alpha^{-21} > 120/\alpha^{*3}$.
Notice that for fixed nonzero $d$ divisible by $q$, the random variables $\mathsf{E}_{\substack{\bar{x}_{q}=t,x\in[N'-2d]}}[f_{3}(x)f_{3}(x+d)f_{3}(x+2d)]$, for $t\in\mathbb{Z}_{q}$, 
are independent. By Hoeffding's inequality, the probability that 
\[
\mathsf{E}_{x\in[N'-2d]}[f_{3}(x)f_{3}(x+d)f_{3}(x+2d)]\ge \alpha^{*}{}^{3} - \frac{\alpha^{*}{}^{3} }{12}
\]
is at most $\exp\left(-2\cdot(12^{-1}\alpha^{*}{}^{3})^2q\right) = \exp(-72^{-1}\alpha^{*}{}^{6}q)$.
Noting that $\alpha^{*}{}^{3}-\alpha^{*}{}^{3}/12 \le \alpha^3(1-\epsilon)$, by the union bound, the probability that there exists a nonzero common
difference $d$ which is divisible by $q$ such that 
\[
\mathsf{E}_{x\in[N'-2d]}[f_{3}(x)f_{3}(x+d)f_{3}(x+2d)]\ge\alpha^{3}(1-\epsilon)
\]
is at most $(N/q)\exp\left(-72^{-1}\alpha^{*}{}^{6}q\right)<q^{4}\exp\left(-72^{-1}\alpha^{*}{}^{6}q\right)<1/2$, as $N<q^5$, $q>\epsilon^{-3}$ and $\epsilon \le \alpha^{7}$. 

If $d$ is not divisible by $q$, each 3-AP with common difference
$d$ occupies three different modulo $q$ classes, and hence the weights
of the elements in the 3-AP are independent random variables. By construction,
for each $x\in[N']$, $\mathbb{E}_{f_{3}}[f_{3}(x)]=f_{2}(x)$. Hence,
by independence, if $(x,x+d,x+2d)\in[N']^3$, 
\[
\mathbb{E}_{f_{3}}[f_{3}(x)f_{3}(x+d)f_{3}(x+2d)]=f_{2}(x)f_{2}(x+d)f_{2}(x+2d)
\]
Thus, 
\[
\mathbb{E}_{f_{3}}\mathsf{E}_{x\in[N'-2d]}[f_{3}(x)f_{3}(x+d)f_{3}(x+2d)]=\mathsf{E}_{x\in[N'-2d]}[f_{2}(x)f_{2}(x+d)f_{2}(x+2d)]\le\alpha^{3}(1-2\epsilon).
\]
For this fixed $d$, we can partition $\mathbb{Z}_{q}$ into five sets $S_{1},S_2,\dots,S_5$ such that for each $i, \, 1\le i\le 5$, the 3-APs $(t,t+\bar{d},t+2\bar{d}),\,t\in S_{i}$ are disjoint,
and $|S_i|\ge q/10$. For each set $S_{i}$, the
random variables $\mathsf{E}_{\substack{\bar{x}_{q}=t,x\in[N'-2d]}
}[f_{3}(x)f_{3}(x+d)f_{3}(x+2d)]$, for $t\in S_{i}$, 
are independent. By Hoeffding's inequality, the probability that 
\begin{align*}
 & \mathsf{E}_{t\in S_{i}}\mathsf{E}_{\substack{\bar{x}_{q}=t,x\in[N'-2d]}
}[f_{3}(x)f_{3}(x+d)f_{3}(x+2d)]\ge\\
 & \mathbb{E}_{f_{3}}\mathsf{E}_{t\in S_{i}}\mathsf{E}_{\substack{\bar{x}_{q}=t,x\in[N'-2d]}
}[f_{3}(x)f_{3}(x+d)f_{3}(x+2d)]+\epsilon\alpha^{3}
\end{align*}
is at most $\exp(-2(\epsilon\alpha^{3})^{2}q/10) = \exp(-5^{-1}\cdot \epsilon^2\alpha^{6}q)$. By the union
bound, the probability that there exists a common difference $d$
not divisible by $q$ with 
\[
\mathsf{E}_{x\in[N'-2d]}[f_{3}(x)f_{3}(x+2d)f_{3}(x+2d)]\ge\alpha^{3}(1-2\epsilon)+\epsilon\alpha^{3}=\alpha^{3}(1-\epsilon)
\]
is at most $5N\exp\left(-5^{-1}\epsilon^{2}\alpha^{6}q\right)<1/2$, where we used  $N<q^{5}$, $q>\epsilon^{-3}$ and $\epsilon\le \alpha^{7}$.

Since $f_{3}(x)=0$ for all $x\notin[N']$, 
\[
\mathsf{E}_{x\in[N-2d]}[f_{3}(x)f_{3}(x+2d)f_{3}(x+2d)]\le\mathsf{E}_{x\in[N'-2d]}[f_{3}(x)f_{3}(x+2d)f_{3}(x+2d)]
\]
Hence, with positive probability, the function $f_{3}$ satisfies the required properties in
Theorem \ref{th:Interval-1}. \end{proof} 

To finish the proof, we prove that the parameters in Step 1 of the construction described at the beginning of the subsection can be chosen. We first prove that we can approximate any large integer with one satisfying the conditions in Theorem \ref{th:lower-product}.

\begin{lemma}\label{lem:approx1} There exist constants $c,\alpha_0>0$ such that if $0\le \alpha \le \alpha_0$, $0 < \epsilon \le \alpha^{7}$, and $r$ is an integer satisfying $\epsilon^{-15} \leq r \leq \tower(c\log(1/\epsilon))$, then we can choose $s\in \left[2,\left\lceil \log_{150}\left(\frac{\epsilon^{-1/4}\alpha^{6}}{8}\right)\right\rceil \right]$
and primes $m_{1},\ldots,m_{s}$ and $q$ satisfying the following properties:\end{lemma} 
\begin{itemize}
\item $\epsilon{}^{-1/3}/2\le m_{1}\le\epsilon{}^{-1/3}$, and, 
\item for $i\ge2$, $n_{i-1}^{6}<m_{i}<\exp(2^{-1} \cdot 64^{-2} \cdot 150{}^{i-1}\epsilon^{1/4}n_{i-1})$ where $n_{i}=\prod_{j=1}^{i}m_{j}$, and,
\item for $n=n_{s}$, we have $n \in[r(1-\epsilon^{2}),r]$. 
\end{itemize}

\begin{proof} 
Choose a sequence of $s$ real numbers $\tilde{m}_{1},\ldots,\tilde{m}_{s}$
satisfying the following properties: $\tilde{m}_1$ is a prime number with $\epsilon{}^{-1/3}/2\le\tilde{m}_{1}\le\epsilon{}^{-1/3}$, $\tilde{n}_{i-1}^{6}<\tilde{m}_{i}<\exp(2^{-1} \cdot 64^{-2}\cdot 150{}^{i-1}\epsilon^{1/4}\tilde{n}_{i-1})$
for $i\ge2$ with $\tilde{n}_{i}=\prod_{j=1}^{i}\tilde{m}_{j}$, and
$\prod_{i=1}^{s}\tilde{m}_{i}=r$. The existence of $\tilde{m}_1$ is guaranteed by Bertrand's postulate. Since $\epsilon^{-15}\le r\le\tower(c\log\frac{1}{\epsilon})$
and $\epsilon\le \alpha^{7}$, if we choose $c,\alpha_0$ sufficiently
small, it is easy to see that there exists a choice of $s$ and $\tilde{m}_{1},\ldots,\tilde{m}_{s}$
satisfying these properties.

For each $i$ let $m_{i}$ be the largest prime such that $m_{i}\le\tilde{m}_{i}$.
So $m_{1}=\tilde{m}_{1}$, and from \cite{BHP}, for all $\tilde{m}_{i}$
large enough, $m_{i}\ge\tilde{m}_{i}-\tilde{m}_{i}^{0.525}$. Hence
\begin{align*}
n & =\prod_{i=1}^{s}m_{i}\ge\tilde{m}_{1}\prod_{i=2}^{s}(\tilde{m}_{i}-\tilde{m}_{i}^{0.525})\\
 & \ge r\prod_{i=2}^{s}(1-\tilde{m}_{i}^{-0.475})\ge r\exp\left(-\sum_{i=2}^{s}\frac{2}{\tilde{m}_{i}^{0.475}}\right)\\
 & \ge r\exp\left(-\frac{4}{\tilde{m}_{2}^{0.475}}\right)\ge r\exp\left(-\epsilon^{2.5}\right)\\
 & \ge r(1-\epsilon^{2.5}),
\end{align*}
where we have applied the inequality $\exp(-x)\ge1-x\ge\exp(-2x)$
for $0\le x\le1/2$, $\sum_{i=2}^{s}\frac{1}{\tilde{m}_{i}^{0.475}}\le\frac{1}{\tilde{m}_{2}^{0.475}}\sum_{i\ge0}\frac{1}{2^{i}}\le\frac{2}{\tilde{m}_{2}^{0.475}}$
as $\tilde{m}_{i+1}\ge\tilde{m}_{i}^{6}\ge2^{1/0.475}\tilde{m}_{i}\ge\ldots\ge2^{(i-1)/0.475}\tilde{m}_{2}$
and $\tilde{m}_{i}\ge\tilde{m}_{2}\ge\epsilon^{-6}$ for $i\ge2$.
Thus we can choose $n=\prod_{i=1}^{s}m_{i}$ and $r(1-\epsilon^{2})\le n\le r$.
\end{proof} 

Using Lemma \ref{lem:approx1}, we prove that the parameters $N',q$ in Step 1 of the construction can be chosen.
\begin{lemma}\label{lem:approx2}
Let $N \ge \epsilon^{-15}$. There exists $q,N'$ such that $N'/q$ is prime, $q$ satisfies the conditions in Theorem \ref{th:lower-product}, and $(1-\epsilon^2)N\le N' \le N$.   
\end{lemma}
\begin{proof} 
We choose $p$ to be a prime number in $(N^{1/5}, \sqrt{\epsilon^4\alpha^3(1-\epsilon)N})$. Let $r = \lfloor N(1-\epsilon^2/4)/p \rfloor \ge N(1-\epsilon^2/2)/p \ge \epsilon^{-7}$. Apply Lemma \ref{lem:approx1} with the above choice of $r$ and with $\epsilon_{7.3} = 4\epsilon$, we find $n$ satisfying the conditions in Theorem \ref{th:lower-product}, applied with $\alpha_{3.4} \in [\alpha,\alpha(1+\epsilon)]$ and $\epsilon_{3.4} \le 4\epsilon$, such that $(1-\epsilon^{2.5})r\le n\le r$. Then we let $N' = pn$ and $q = n$. 

We have $$N' = pn \ge (1-\epsilon^{2.5})rp \ge (1-\epsilon^2/2)(1-\epsilon^{2.5})N \ge (1-\epsilon^2)N,$$ thus $(1-\epsilon^2)N\le N'\le N$, finishing the proof. 
\end{proof}

\appendix
\section{Proof of Theorem \ref{th:lower-Z} from Theorem \ref{th:Interval-1}}\label{appendix:lower}

In this appendix, we show how to deduce Theorem \ref{th:lower-Z} from its functional version Theorem \ref{th:Interval-1}. 

We first show that Theorem \ref{th:lower-Z} holds if it holds in the case $N\ge \epsilon^{-15}$. Assume that $N<\epsilon^{-15}$. Recall that $N(\alpha)$ denotes the least positive integer such that if $N \ge N(\alpha)$ then any $A\subset [N]$ with $|A|\ge \alpha N$ contains a nontrivial 3-AP. If $N<N(\alpha)$, there exists a subset $A$ of $[N]$ with $|A|\ge \alpha N$ and $A$ does not contain a nontrivial 3-AP. In this case, for all $d\ne 0$, $$\mathsf{E}_{x\in [N-2d]}[A(x)A(x+d)A(x+2d)]=0,$$ so the conclusion of Theorem \ref{th:lower-Z} holds. By Behrend's bound (\cite{B}, \cite{El}), we have that $N(\alpha) \ge  \exp((\log 1/\alpha)^2/6)$ for $\alpha>0$ sufficiently small. For $N\ge N(\alpha)\ge \exp((\log 1/\alpha)^2/6)$, let $\epsilon_{0}=N^{-1/15}$
so $N=\epsilon_{0}^{-15}$. We have $\epsilon_{0}>\epsilon$ and $\epsilon_{0}\le \exp(-(\log 1/\alpha)^2/90) \le \alpha^{12}$ (assuming $\alpha_0$ is small enough). By choosing $\alpha_0$ in Therem \ref{th:lower-Z} small enough, we may assume that for all $x\le \alpha_0$,
$x^{-15}<\tower(c\log(1/x))$. Then $N=\epsilon_{0}^{-15}<\tower(c\log(1/\epsilon_0))$
as $\epsilon_{0}<\alpha \le \alpha_0$. Thus, $\epsilon_{0}^{-15}\le N\le\tower(c\log(1/\epsilon_0))$, so we can apply Theorem \ref{th:lower-Z} with $\epsilon_0$ in place of $\epsilon$ to obtain the desired set $A$. The same argument also shows that we only need to prove Theorem \ref{th:Interval-1} when $N\ge \epsilon^{-15}$.

We next discuss how to obtain a set $A$ with the properties in Theorem \ref{th:lower-Z} from Theorem \ref{th:Interval-1} when $N\ge \epsilon^{-15}$. This follows via a standard sampling argument which is essentially similar to Lemma 9 in \cite{FPI}. However, there are some small differences to the argument which we now highlight. Given a function $f:[N]\to [0,1]$ such that the density of 3-APs with common difference $d$ of $f$ is small for all $0 < d<N/2$, we sample a set $A$ where each element $x\in [N]$ is in $A$ with probability $f(x)$ independent of each other. If the density of 3-APs with common difference $d$ in $A$ is concentrated around its expectation, which is the density of 3-APs with common difference $d$ of $f$, then it is small with high probability. However, for $d$ near $N/2$, there are very few 3-APs with common difference $d$, and we do not have sufficiently strong concentration to be able to take a union bound over all such $d$. To get around this, we define a function $f'$ which is $0$ for all $x$ close to $N$, and which is equal to $f$ elsewhere, and sample the set $A$ from $f'$. This ensures that for common differences $d$ which are close to $N/2$, set $A$ contains very few 3-APs with common difference $d$.

We now carry out the details. By Theorem \ref{th:Interval-1} applied with $\alpha$ replaced by $\alpha+2\epsilon$ and $\epsilon$ replaced by $12\epsilon/\alpha^3 \le \alpha^7$, we can find a function $f:[N]\to [0,1]$ such that for any $0<d<N/2$, $$\mathsf{E}_{x\in[N-2d]}[f(x)f(x+d)f(x+2d)]\le (\alpha+2\epsilon)^{3}(1-12\epsilon/\alpha^3).$$ Define $f':[N]\to [0,1]$ by $f'(x)=0$ if $x\ge N(1-\epsilon)$ and $f'(x)=f(x)$ otherwise. We define $A$ to be a random subset of $[N]$ where each $x\in [N]$ is in $A$ with probability $f'(x)$, independently of the other elements. In particular, $A$ does not contain $x$ if $x\ge N(1-\epsilon)$. Hence, if the common difference $d$ is larger than $N(1-\epsilon)/2$ then $$\mathsf{E}_{x\in [N-2d]}[A(x)A(x+d)A(x+2d)] = 0.$$ If $d$ is at most $N(1-\epsilon)/2$, then $N-2d \ge \epsilon N$, so by Hoeffding's inequality, with probability at least $1-\exp(-\epsilon^2(N-2d)) \ge 1-\exp(-\epsilon^3 N)$, $$|\mathsf{E}_{x\in [N-2d]}[A(x)A(x+d)A(x+2d)]-\mathsf{E}_{x\in [N-2d]}[f''(x)f''(x+d)f''(x+2d)]| \le \epsilon.$$ Furthermore, note that \begin{align*}
    \mathsf{E}_{x\in [N-2d]}[f''(x)f''(x+d)f''(x+2d)] &\le \mathsf{E}_{x\in[N-2d]}[f(x)f(x+d)f(x+2d)]\\ &\le (\alpha+2\epsilon)^3(1-12\epsilon/\alpha^3)\\ &\le \alpha^3-2\epsilon.
\end{align*} By Hoeffding's inequality, with probability at least $1-\exp(-\epsilon^2 N)$, the density of $A$ is at least $$\mathsf{E}_x[f'(x)] - \epsilon \ge \mathsf{E}_x[f(x)] - \frac{\epsilon N}{N} - \epsilon = \alpha.$$ Thus, by the union bound, with probability at least $1-N\exp(-\epsilon^3 N)-\exp(-\epsilon^2 N)$, $A$ is a set with density at least $\alpha$ such that for all $0<d<N/2$, $$\mathsf{E}_{x\in [N-2d]}[A(x)A(x+d)A(x+2d)] \le \alpha^3 - \epsilon.$$ Since $N > \epsilon^{-15}$, for $\epsilon$ sufficiently small, $1-N\exp(-\epsilon^3 N)-\exp(-\epsilon^2 N) > 0$. This gives Theorem \ref{th:lower-Z}, assuming Theorem \ref{th:Interval-1}. 

\end{document}